\documentclass[11pt]{amsart}
\usepackage{a4wide}
\usepackage[T1]{fontenc}
\usepackage{amssymb,amsmath,amsthm,latexsym}
\usepackage{mathrsfs}
\usepackage[usenames,dvipsnames]{color}
\usepackage{euscript}
\usepackage{graphicx}
\usepackage{mdwlist}
\usepackage{enumerate}
\usepackage{mathtools,dsfont,wasysym}
\usepackage{bbm}
\usepackage{stmaryrd}
\usepackage{centernot}
\usepackage{todonotes}

\makeatletter
\@namedef{subjclassname@2020}{\textup{2020} Mathematics Subject Classification}
\makeatother

\newtheorem{theorem}{Theorem}[section]
\newtheorem{lemma}[theorem]{Lemma}
\newtheorem{corollary}[theorem]{Corollary}
\newtheorem{proposition}[theorem]{Proposition}

\newcounter{maintheorem}

\theoremstyle{remark}
\newtheorem{remark}[theorem]{Remark}
\theoremstyle{definition}
\newtheorem{definition}[theorem]{Definition}

\newtheorem{example}[theorem]{Example}
\numberwithin{equation}{section}
\makeatother

\newcommand{\nn}[1]{{\left\vert\kern-0.25ex\left\vert\kern-0.25ex\left\vert #1
\right\vert\kern-0.25ex\right\vert\kern-0.25ex\right\vert}}

\renewcommand{\leq}{\leqslant}
\renewcommand{\geq}{\geqslant}

\DeclareMathOperator{\Ric}{Ric}

\DeclareMathOperator{\Vol}{Vol}

\newcounter{smallromans}

{\end{list}}

\begin{document}
\title[]{The porous medium equation on noncompact manifolds with nonnegative Ricci curvature:\\ a Green function approach}

\author[G.~Grillo]{Gabriele Grillo}
\address[G.~Grillo]{Politecnico di Milano, Dipartimento di Matematica, Piazza Leonardo da Vinci 32, 20133 Milano, Italy}
\email{gabriele.grillo@polimi.it}

\author[D.D.~Monticelli]{Dario D. Monticelli}
\address[D.D.~Monticelli]{Politecnico di Milano, Dipartimento di Matematica, Piazza Leonardo da Vinci 32, 20133 Milano, Italy}
\email{dario.monticelli@polimi.it}

\author[F.~Punzo]{Fabio Punzo}
\address[F.~Punzo]{Politecnico di Milano, Dipartimento di Matematica, Piazza Leonardo da Vinci 32, 20133 Milano, Italy}
\email{fabio.punzo@polimi.it}


\keywords{Porous medium equation, smoothing estimates, Green function, Riemannian manifolds, Ricci curvature, potential estimates}
\subjclass[2020]{35K55, 35D45, 47G40, 31C12, 58J35}

\begin{abstract} We consider the porous medium equation (PME) on complete noncompact manifolds $M$ of nonnegative Ricci curvature. We require nonparabolicity of the manifold and construct a natural space $X$ of functions, strictly larger than $L^1$, in which the Green function on $M$ appears as a weight, such that the PME admits a solution in the weak dual (i.e. potential) sense whenever the initial datum $u_0$ is nonnegative and belongs to $X$. Smoothing estimates are also proved to hold both for $L^1$ data, where they take into account the volume growth of Riemannian balls giving rise to bounds which are shown to be sharp in a suitable sense, and for data belonging to $X$ as well.
\end{abstract}
\maketitle


\section{Introduction, preliminaries and statements of the main results}

The study of nonlinear diffusion of porous medium/fast diffusion type, namely of the problem
\begin{equation}\label{eucl}
\begin{cases}
\partial_t u - \Delta(u^m) = 0 & \text{ in } \; \mathbb R^n\times (0, \infty),\\
u = u_0  & \text{ in } \; \mathbb R^n\times \{0\}\,.
\end{cases}
\end{equation}
and related variants, where $m>0$, has a long history in the literature and is a very active topic till today. Since it is hopeless to give precise references on this wide topic, we limit ourselves to quote e.g. the monograph \cite{V}, which gives a fairly complete account of the results up to the date of its publication, and e.g. \cite{BF, BF2} and references quoted therein for some of the more recent work. Equation \eqref{eucl}, when $m>1$ as will be the case in the present work and in which the equation is labeled \it porous medium equation \rm (PME for short), is a nonlinear analogue of the heat equation in which the diffusion coefficient, or \it pressure\rm, is $v=mu^{m-1}$, hence a quantity which vanishes when $u$ does, and as a consequence it can be shown that compactly supported data correspond to solutions that are compactly supported for all times, in sharp contrast with solutions to the heat equation. On the other hand, there are also significant similarities between \eqref{eucl} and the heat equation as concerns other properties. For example, solutions to \eqref{eucl} are well defined for data in $L^p(\mathbb R^n)$ for all $p\in[1,+\infty]$, $L^p$ norms are nonincreasing along the evolution, the parabolic comparison principle holds, and \it smoothing effects \rm are valid, in the sense that $L^p$ ($p\in[1,+\infty)$) data give rise to \it bounded \rm solutions for all $t>0$, with quantitative bounds on the $L^\infty$ norm of the solution in terms of $t$ and of the $L^p$ norm of the initial datum.

On the other hand, the heat equation can be posed on a Riemannian manifold, and its study in that context is a fascinating topic, still being developed. Again, since it is hopeless to give an account for the impressive amount of research on such topic, we refer to \cite{Gri,Gri2, Gri3} and references quoted therein as in-depth surveys. It is worth noticing at least two very different situations, as concerns the heat kernel behaviour for large times. On the one hand, when sectional curvature is \it strictly negative\rm, the departure of the heat kernel behaviour from the Euclidean one is impressive. For example,
\[
K_{\mathbb H^n}(t,x,x)\underset{t\to+\infty}{\sim}\frac c{t^{3/2}}e^{-\frac{(n-1)^2}4t}
\]
for a suitable constant $c>0$, where $\mathbb H^n$ is the \it hyperbolic space \rm of dimension $n$ and $K_{\mathbb H^n}(t,x,x)$ its diagonal heat kernel, for $t>0, x\in\mathbb H^n$. This has to be compared with Euclidean behaviour
\[
K_{\mathbb R^n}(t,x,x)\underset{t\to+\infty}{\sim}\frac c{t^{n/2}},
\]
showing the dramatic change occurring. On the other hand, in the case of manifolds $M$ with \it nonnegative Ricci curvature \rm both similarities and differences are present, in fact it can be shown that
\[
K_{\mathbb M}(t,x,x)\asymp\frac c{\textrm{Vol}\,(B_{\sqrt t}(x))},
\]
where $B_s(x)$ is the Riemannian ball centered at $x\in M$ and of radius $s$, and Vol denotes the Riemannian volume. One recovers the Euclidean asymptotic behaviour in the case of manifolds with \it maximal volume growth\rm, but in general the long time decay can be slower than the Euclidean one. See \cite{LY} for this fundamental result.

In the nonlinear setting, a number of recent works have dealt with the analogue of \eqref{eucl} in the Riemannian setting, see e.g. \cite{BGV, GM1, GMP, GMP2, GMV, GMV2}. In fact, a series of new phenomena, significantly different from the ones occurring in the Euclidean case, occur in the nonlinear setting, similarly to what happens in the linear one, for example as concerns the asymptotic behaviour of solutions and the class of admissible initial data and initial traces. We also mention \cite{AT, AT2}, where \it weighted \rm diffusion equations of doubly nonlinear type are studied through different methods, which crucially require information on the isoperimetric profile of the manifold and apply, as stated, to equations with weights \it vanishing at infinity\rm. The nonnegatively curved case seems so far not studied at all, and it is our goal here to commence this study.

\smallskip\noindent
It is important to comment that the recent analysis of \cite{BBGM}, which improves the particular case dealt with in \cite{BBGG} and also uses potential methods for fractional nonlinear diffusion on manifolds which satisfy the weaker bound Ric$\,\ge-k$ for some $k\ge0$, is constrained to the validity of \it additional functional inequalities\rm, namely the (Euclidean-type) Faber-Krahn inequality, or equivalently the (Euclidean-type) Sobolev inequality, which of course typically does not hold in our setting. The validity of the Euclidean-type Faber-Krahn inequality is a very strong assumption, which implies Euclidean-type bounds for the heat kernel and the Green function, which \it need not hold \rm in our setting and are both crucial for the analysis of \cite{BBGM}, whose results are essentially confined to the case of \it nonpositively curved \rm manifolds. Therefore, the present analysis is entirely new and requires new methods and ideas, though some technical tools used in \cite{BBGM} are used also here. Here our main assumptions, besides the nonnegative Ricci one, concern \it volume growth of Riemannian balls\rm, see \eqref{noncollapsing}, \eqref{uniformvolume}, \eqref{integrablef}, which are qualitatively (mild) uniformity assumptions on such growth. These assumptions are crucial for our proofs, and also appear naturally in the proof that the space $L^1_G(M)$, which is the maximal one in which our results hold, is \it strictly larger than \rm $L^1(M)$, see Proposition \ref{Gprop}.
Note that in  \cite{BBGM}, in view of the Faber-Krahn inequality, the volume of balls is controlled from above by the Euclidean one if the radius is small enough, on the other hand, it is controlled from below by the Euclidean one for any radius. Therefore, such volume conditions are very different from our assumptions  \eqref{noncollapsing}, \eqref{uniformvolume}, \eqref{integrablef}. We also comment that 
in \cite{BBGM} the parabolic Harnack inequality and heat kernel estimates are exploited in various arguments, whereas the present geometrical setting allows to obtain directly the necessary estimates on the Green function and their use, see e.g. Lemma \ref{Greenest} and Proposition \ref{stimapr}.

\medskip
We shall consider here the problem:
\begin{equation}\label{e23f}
\begin{cases}
\partial_t u - \Delta(u^m) = 0 & \text{ in } \; M\times (0, \infty),\\
u = u_0  & \text{ in } \; M\times \{0\}\,,
\end{cases}
\end{equation}
where $m>1$ and $M$ is a complete, noncompact Riemannian manifold of nonnegative Ricci curvature, satisfying some further assumptions discussed in Section \ref{ass} below, which qualitatively amount to assuming non-collapse of the volume function and a further uniformity condition on volume. We collect in Section \ref{ex} a number of explicit cases in which such assumptions are satisfied. Besides, it will be crucial to assume that $M$ is \it nonparabolic\rm, i.e. that it admits a minimal, positive Green function $G(x,y)$, this being  equivalent to the fact that the (minimal) heat kernel $K(t,x,y)$ is integrable as $t\to+\infty$ for any $x,y\in M$, $x\not=y$, as in that case one has
\[
G(x,y)=\int_0^{+\infty}K(t,x,y)\,{\textrm d}t,\ \ \ \forall x\not=y, x,y\in M.
\]
In fact, we shall use a method that crucially depends on the existence of the Green function, and that has been introduced and exploited e.g. in \cite{BFR, BSV, BV2, BV, BV1, V3}. In it, the qualitative idea is considering an equation for the \it potential \rm of the solution, $U=(-\Delta)^{-1}u$, which \it formally \rm solves $U_t=-u^m$, so that $U$ is decreasing if $u\ge0$. This formal procedure can be made mathematically consistent, see Section \ref{wds}, and will be our main functional analytic tool to reach our goals. Our main results can be summarized as follows:
\begin{itemize}
\item Construction of a suitable \it weighted \rm space, crucially associated to the Green function $G$, see \eqref{e24f}-\eqref{19}, on which weak dual solutions, in the sense precisely given in Section \ref{wds}, exist, see Theorem \ref{exists}. We show in Section \ref{L1G} that such space is always \it strictly larger \rm than $L^1(M)$; in particular, for example, on manifolds with polynomial volume growth at infinity, compatible with the nonparabolicity assumption, it includes e.g. functions of the type $f(x)\asymp d(o,x)^{-a}$ at infinity, where $o$ is a given reference point in $M$ and $d$ is the Riemannian distance, for all $a>2$, hence with a power bound \it independent of the dimension\rm.
    \item A \it smoothing estimate \rm for solutions corresponding to $L^1$ data, see Theorem \ref{decay}, namely a quantitative bound for the $L^\infty$ norm of the solutions at time $t$ in terms of time and of the $L^1$ norm of the datum. The bound is \it Euclidean \rm for small times, and takes carefully into account the volume growth of Riemannian balls for large times. We will show in Section \ref{opt} that the bound is sharp in the sense that there are manifolds satisfying the running assumption in which the function of time appearing in Theorem \ref{decay} cannot be improved, since a matching lower bound holds.
    \item A different \it smoothing estimate \rm for solutions corresponding to $L^1_G(M)$ data, see Theorem \ref{decayL1G}. While we do not know whether the long-time behaviour is sharp, we find anyway relevant that a smoothing effect holds also for the larger class of data considered.

\end{itemize}

The Green function approach in the Euclidean setting has been recently developed in the foundational work \cite{BE} in very great generality, including the treatment of non-local operators in $\mathbb R^n$. However, in \cite{BE} power-like integral conditions on the Green function are assumed; they need not to be fulfilled in our framework (see \eqref{Greenest}, \eqref{14}, \eqref{5} below).

To be more precise, we comment that in \cite{BE}, see in particular their condition ($G_1$) and its generalizations given e.g. in Theorem 3.3 and Example 7.11, it is possible to deal with cases in which the quantity $I(R)=\int_{B_R(x_0)}G(x,x_0)\textrm{d}\mu$, technically crucial both in \cite{BE} in the Euclidean setting and here,  behaves like \it different \rm powers $R^a$ in the ranges $R\to0$ and $R\to+\infty$, as this might be e.g. related to operators which are of the form $-\Delta+(-\Delta)^s$ for $s\in(0,1)$. In our case, see \eqref{14}, \eqref{5} below, the small $R$ behaviour of $I(R)$ is, as expected, \it always Euclidean \rm in the sense that $I(R)\asymp R^2$ for $R<1$, whereas the large $R$ behaviour is essentially \it arbitrary \rm and in particular not polynomial at all, as \eqref{5} clearly shows. In fact, such behaviour is strictly related to the geometry of $M$, this being the main difference with the Euclidean case. As a consequence, the long time smoothing effects may have in principle unpredictable behaviours, see \eqref{7} and e.g. Corollary \ref{almosteuc}(b).

Besides, in \cite{BE} smoothing effects are proved also under their weaker condition ($G_1'$), that does not require any decay condition at $R\to+\infty$  on $I(R)$, just boudndedness as a function of $R$. The resulting bounds are however \it weaker \rm then ours under two respects. On the one hand, their bound for $t\ge1$ under such assumption is
\begin{equation*}
 \|u(t)\|_{L^\infty(M)}\leq\frac{C}{t^{\frac{1}{m}}}\|u_0\|_{L^1(M)}^\frac{1}{m},\ \ \ \forall u_0\in L^1(M)
 \end{equation*}
so that the predicted time decay is \it slower \rm than the one we prove in Theorem \eqref{results} and, e.g., in Corollary \ref{almosteuc}(a). On the other hand, our results hold in $L^1_G(M)$, which we prove to be strictly larger than $L^1(M)$. Remarkably, the time decay we prove in Theorem \ref{decayL1G} is instead \it the same \rm proved in \cite{BE}, but we prove that it holds in the \it larger \rm space $L^1_G(M)$. In a sense, such long time behaviour is the worst possible, see also Remark \ref{worst} below.

The geometric setting considered here makes it natural to consider different type of behaviour of the Green function and its integrals on (Riemannian) balls, so that we consider our work as a natural continuation of \cite{BE}, where the focus is not anymore on considering different operators on a fixed space, namely $\mathbb R^n$, but in dealing with a given equation on spaces whose geometric properties vary, as reflected also in the Green function's behaviour. We also comment that the space $L^1_G$ is \it not considered \rm in \cite{BE}.

We now introduce our results by stating precisely our assumptions in Section \ref{ass}, then defining the concept of weak dual solution in Section \ref{wds}, and finally stating them in Section \ref{results}.

\subsection{Assumptions}\label{ass}
Let $(M,g)$ be a $n$-dimensional Riemannian manifold, we denote by $\operatorname{Ric}$, $d\mu$, $\nabla$ and $\Delta$ the Ricci curvature tensor, the Riemannian measure, the gradient and the Laplace-Beltrami operator on $M$, respectively. We will denote by $B_r(x)$ the geodesic ball centered at $x\in M$ with radius $r>0$ and by $d(x,x_0)$ the geodesic distance between $x,x_0\in M$.

Throughout the paper we will always assume that $M$ has nonnegative Ricci curvature, $\operatorname{Ric}\geq0$, that it is complete, noncompact and that it is \textit{non-collapsing}, i.e.
\begin{equation}\label{noncollapsing}
  \alpha:=\inf_{x\in M}\operatorname{Vol}(B_1(x))>0.
\end{equation}

Note that there exist manifolds with Ric$>0$ which are nonparabolic and for which \eqref{noncollapsing} does \it not \rm hold, see e.g. \cite{CK}.
We will also always make the following \textit{uniformity assumption }on the rate of growth of the Riemannian volume of geodesics balls with large radii, with respect to their centers: there exist $\gamma>0$, $R_0\geq1$ and a $C^1$, nondecreasing function $f:[R_0,+\infty)\rightarrow(0,+\infty)$ such that 
\begin{equation}\label{uniformvolume}
  \frac{Rf(R)}{\Vol(B_R(x))}\leq\gamma\frac{rf(r)}{\Vol(B_r(x))}\qquad\text{ for all }R\geq r\ge R_0, x\in M
\end{equation}
and
\begin{equation}\label{integrablef}
  \beta:=\int_{R_0}^{+\infty}\frac{1}{f(t)}\,{\rm d}t<+\infty.
\end{equation}

\subsection{The space $L^1_G(M)$ and Weak Dual Solutions}\label{wds}

Let $(M,g)$ be a non-parabolic Riemannian manifold and let $G$ be its minimal positive Green function. For each $x_0\in M$, we set
\begin{equation}\label{e24f}
\|f\|_{L^1_{x_0, G}}:=\int_{B_1(x_0)}|f(x)|\, d\mu(x) + \int_{M\setminus B_1(x_0)}|f(x)| G(x, x_0)\, d\mu(x)\,,
\end{equation}
and we introduce the following {\it weighted space}, defined in terms of $x_0$ and of the Green function $G$:
\begin{equation*}\label{Gx}L^1_{x_0, G}(M):=\left\{f:M\to \mathbb R \text{ measurable } : \|f\|_{L^1_{x_0, G}}<\infty \right\}\,.\end{equation*}
Furthermore, we consider the following {\it weighted space}, by removing the dependence on $x_0$:
\begin{equation*}\label{G}L^1_{G}(M):=\left\{f:M\to \mathbb R \text{ measurable } : \sup_{x_0\in M}\|f\|_{L^1_{x_0, G}}<\infty \right\}\,.\end{equation*}
For any $f\in L^1_{G}(M)$ we define
\begin{equation*}\label{19}
  \|f\|_{L^1_{G}}:=\sup_{x_0\in M}\|f\|_{L^1_{x_0, G}}\,.
\end{equation*}
Observe that we shall prove later that, for \it any \rm manifold satisfying our running assumptions, one has \[L^1(M)\subsetneq L^1_G(M)\,.\]
 We shall discuss some properties of the space $L^1_G(M)$ later on, see Proposition \ref{Gprop}.

\smallskip We now provide the definition of \it weak dual solutions\rm, following \cite{BV2, BV} in the Euclidean case, and \cite{BBGM} on manifolds.

\begin{definition}\label{defsol}
Let $u_0\in L^1_G(M), u_0\geq 0$. A nonnegative measurable function $u$ is said to be a {\em Weak Dual Solution} (WDS for short) to problem \eqref{e23f} if, for any $T>0$,
\begin{itemize}
\item for any $x_0\in M, \, u\in C([0, T]; L^1_{x_0, G}(M))$
\item $u^m\in L^1((0, T); L^1_{\textrm{loc}}(M))$;
\item $u$ satisfies the equality
\begin{equation}\label{e25f}
\int_0^T\int_M \partial_t \psi \, (-\Delta)^{-1}u \, d\mu dt -\int_0^T\int_M u^m\, \psi\, d\mu dt =0
\end{equation}
for any test function $\psi\in C^1_c((0, T);L^\infty_{\textrm{loc}}(M))$;
\item $u(\cdot, 0)=u_0$ a.e. in $M$\,.
\end{itemize}
\end{definition}

Note that if $u$ is a WDS to problem \eqref{e23f} in the sense of Definition \ref{defsol}, then $(-\Delta)^{-1}u\in C^0([0, T]; L^1_{\textrm{loc}}(M))$ (see the proof of Proposition \ref{prop53}).

\subsection{Main Results}\label{results}

Concerning the existence of WDS to problem \ref{e23f} we prove what follows, in the $L^1_G(M)$ setting.
 \begin{theorem}[Existence of WDS for nonnegative initial data in $L^1_G(M)$]\label{exists}
  Let $M$ be a Rie\-man\-nian manifold with $\Ric\geq0$ satisfying \eqref{noncollapsing}, \eqref{uniformvolume} and \eqref{integrablef}. For any nonnegative initial datum $u_0\in L^1_G(M)$ there exists a weak dual solution to problem \eqref{e23f}, in the sense of Definition \ref{defsol}.
\end{theorem}

\begin{remark}
  The WDS $u$ in Theorem \ref{exists} is constructed as a monotone limit of mild $L^1(M)\cap L^\infty(M)$ solutions, and it is easy to show that it does not depend on the particular choice of the monotone sequence $u_{0,k}\in L^1(M)\cap L^\infty(M)$ approximating the initial datum $u_0\in L^1_G(M)$. The result can be obtained in a similar way as \cite[Theorem 2.4]{BBGM}, therefore we will omit the proof.
\end{remark}

Depending on whether the initial datum $u_0$ belongs to $L^1(M)$ or $L^1_G(M)$, we obtain two different smoothing estimates. They are described in the following two results.

\begin{theorem}[Smoothing for $L^1$ data]\label{decay}
 Let $M$ be a Riemannian manifold with $\Ric\geq0$ satisfying \eqref{noncollapsing}, \eqref{uniformvolume} and \eqref{integrablef}. Let
 \begin{eqnarray*}
    h(R) &=& Rf(R)\int_R^{+\infty}\frac{1}{f(t)}\,{\rm d}t+R^2 \\
    F(R) &=& \inf_{x\in M} \Vol(B_R(x))
 \end{eqnarray*}
 and $\theta(R)=F(R)(h(R))^\frac{1}{m-1}$, 
 $f$ as in \eqref{uniformvolume}. There exists $C>0$ such that for any initial datum $u_0\in L^1(M)$, if $u$ is a WDS of
  \eqref{e23f}, then
 \begin{equation}\label{7}
 \|u(t)\|_{L^\infty(M)}\leq\frac{C}{t^{\frac{1}{m-1}}}\left(h\left(\theta^{-1}\left(t^{\frac{1}{m-1}}\|u_0\|_{L^1(M)}\right)\right)\right)^\frac{1}{m-1}
 \end{equation}
  for every 
  $t\ge K\|u_0\|_{L^1(M)}^{-(m-1)}$, with $K=\theta(R_0)^{m-1}$ and $R_0$ as in \eqref{uniformvolume}.
  Moreover there exists $C>0$ such that
 \begin{equation}\label{13}
 \|u(t)\|_{L^\infty(M)}\leq\frac{C}{t^{\frac{n}{(m-1)n+2}}}\|u_0\|_{L^1(M)}^\frac{2}{n(m-1)+2}
 \end{equation}
  for every $0<t\leq K\|u_0\|_{L^1(M)}^{-(m-1)}$. 
\end{theorem}

\begin{remark}\label{lt}
When $M=\mathbb R^n$, it is well known that \eqref{13} holds for all $t>0$ and that such bound is sharp \it both for short and for long times\rm, see e.g. \cite{V, V4}, since such bound is attained as an equality by the so-called Barenblatt solution, which is the (unique) solution of prescribed $L^1(\mathbb R^n)$ norm $A$, for all $t>0$, corresponding to the initial datum $A\delta_0$, $\delta_0$ being the Dirac delta centered at the origin.

\end{remark}

To give an idea of the kind of explicit results we get from our results above, we state the following Corollary.

\begin{corollary}\label{almosteuc}


 Let $M$ be a Riemannian manifold with $\Ric\geq0$ satisfying \eqref{noncollapsing}, \eqref{uniformvolume} and \eqref{integrablef}.
  \begin{itemize}
\item[(a)] If $f(R)=R^{k-1}(\log R)^\delta$ for some $k\in(2,n]$, $\delta\in \mathbb{R}$ and  \begin{equation*}\Vol(B_R(x))\geq CR^\lambda\end{equation*} for some $\lambda\in(2,n]$ and every large enough $R>0$, $x\in M$, then
    \begin{equation*}\|u(t)\|_{L^\infty}(M)\leq \frac{C}{t^\frac{\lambda}{(m-1)\lambda+2}}\|u_0\|_{L^1(M)}^\frac{2}{(m-1)\lambda+2},\end{equation*}
  for all large enough $t>0$. \\

  \item[(b)] If $f(R)=R(\log R)^\delta$ for some $\delta>1$ and \begin{equation*}\Vol(B_R(x))\geq CR^\lambda(\log R)^\sigma\end{equation*} for some $\lambda\in[2,n]$, $\sigma\in \mathbb{R}$, and every large enough $R>0$, $x\in M$, then
  \begin{equation*}\|u(t)\|_{L^\infty(M)}\leq \frac{C}{t^\frac{1}{m-1}}\left(G\big(t^\frac{1}{m-1}\|u_0\|_{L^1(M)}\big)\right)^\frac{2}{m-1}\left(\log\left(G\big(t^\frac{1}{m-1}\|u_0\|_{L^1(M)}\big)\right)
  \right)^\frac{1}{m-1},\end{equation*}
  for all large enough $t>0$, where

$$G(s)=\begin{cases}e^{\frac{b}{a}W_0\left(\frac{a}{b}s^\frac{1}{b}\right)}&\textrm{if } b\neq0\\ s^\frac{m-1}{(m-1)\lambda+2}&\textrm{if }b=0\end{cases}$$
 with
 $$a=\lambda+\frac{2}{m-1},\qquad b=\sigma+\frac{1}{m-1}$$
 and where $W_0$ is 
 the inverse on $[-1,\infty)$ of the function $w(x)=xe^x$. \end{itemize}\end{corollary}

\medskip
In view of Remark \ref{lt}, the result of point (a) above can be described by saying that the power appearing in the assumed lower bound for the volume growth acts as an \it effective dimension for large times \rm for the evolution we consider, independently of $k$ and $\delta$. In particular if $\lambda=n$, and thus $f(R)=R^{n-1}$, we recover the Euclidean rate of decay. In turn, the result of point (b) shows that the smoothing effect can be considerably more complicated when the geometry of $M$ is less close to the Euclidean one. It this second case the relation between volume growth and the decay estimate is only implicit and in no case, including the special one $b=0$, it resembles the Euclidean situation. In particular, no effective dimension can be identified.

\begin{remark}\label{rem2}
  Since $F,h$ are strictly increasing and diverging to $+\infty$ as $R$ tends to $+\infty$, see Remark \ref{15}, $\theta^{-1}$ is well defined on $[\theta(R_0), \infty)$ and it diverges to $+\infty$ as $R$ tends to $+\infty$.

 Note also that, as it is clear from the proof, estimate \eqref{7} holds for any function $\theta$ which is strictly increasing, diverging to $+\infty$ when $R$ tends to $+\infty$ and such that for some constant $c>0$ one has $\theta(R)\leq c F(R)(h(R))^\frac{1}{m-1}$.
\end{remark}

\begin{theorem}[Smoothing for $L^1_G(M)$ data]\label{decayL1G}
 Let $M$ be a Riemannian manifold with $\Ric\geq0$ satisfying \eqref{noncollapsing}, \eqref{uniformvolume} and \eqref{integrablef}. There exists $C>0$ such that for any initial datum $u_0\in L^1_G(M)$, if $u$ is a WDS of problem \eqref{e23f} with $m>1$,
  \begin{equation}\label{20}
 \|u(t)\|_{L^\infty(M)}\leq\frac{C}{t^{\frac{1}{m}}}\|u_0\|_{L^1_G(M)}^\frac{1}{m}
 \end{equation}
  for every $t\geq\|u_0\|_{L^1_G(M)}^{-(m-1)}$. Moreover there exists $C>0$ such that
  \begin{equation}\label{21}
 \|u(t)\|_{L^\infty(M)}\leq\frac{C}{t^{\frac{n}{(m-1)n+2}}}\|u_0\|_{L^1_G(M)}^\frac{2}{n(m-1)+2}
 \end{equation}
  for every $0<t\leq\|u_0\|_{L^1_G(M)}^{-(m-1)}$.
\end{theorem}

\subsection{Organization of the paper.} In Section \ref{prel}, we discuss our geometric assumptions and present a number of examples of manifolds which satisfy them. Besides, we analyze some properties of the space $L^1_G(M)$. In Section \ref{apriori} we show some estimates for the Green function, for the potential of test functions and for WDS to \eqref{e23f} with $u_0\in L^1(M)\cap L^\infty(M).$

In Section \ref{approx} we show the existence of such type of solutions, while in Section \ref{basic} we obtain some basic estimates on them. The main results are proved in Section \ref{proofmain}. Finally, in Section \ref{optimal} we apply our general results to some classes of manifolds and we show optimality of our decay estimate for large times.

\section{Geometric setting and examples}\label{prel}
We recall that a Riemannian manifold is called non-parabolic if the Laplace-Beltrami operator $\Delta$
 on M has a positive fundamental solution, and parabolic otherwise. If $M$ is non-parabolic, we denote by $G(x,x_0)$, $x,x_0\in M$ with $x\neq x_0$, its minimal positive Green function.

We start with some results concerning the volume of Riemannian balls, that will be crucial in the proof of Proposition \ref{stimapr}

\subsection{On volume bounds}\label{rem1}
  It is well known, by combining the results of \cite{LY,LT}, that if $\Ric\geq0$ then $M$ is non-parabolic if and only if
 $$
 \int_{r}^{+\infty}\frac{t}{\Vol(B_t(x))}\,{\rm d}t<+\infty
 $$
 for some, and hence for all, $x\in M$ and $r>0$. In this case there exist $c_1,c_2>0$ such that
 \begin{equation}\label{Green}
   c_1\int_{d(x,x_0)}^{+\infty}\frac{t}{\Vol(B_t(x_0))}\,{\rm d}t\leq G(x,x_0)\leq c_2\int_{d(x,x_0)}^{+\infty}\frac{t}{\Vol(B_t(x_0))}\,{\rm d}t.
 \end{equation}

We note that if \eqref{uniformvolume} and \eqref{integrablef} hold, then $M$ is non-parabolic. Indeed for every $x_0\in M$ we have
\begin{equation*}
  \int_{R_0}^{+\infty}\frac{t}{\Vol{(B_t(x_0))}}\,{\rm d}t =\int_{R_0}^{+\infty} \frac{tf(t)}{\Vol(B_t(x_0))} \frac{1}{f(t)}\,{\rm d}t \leq \beta\gamma\frac{R_0f(R_0)}{\Vol(B_{R_0}(x_0))}<+\infty.
\end{equation*}

  We also comment that by Bishop-Gromov volume comparison theorem (see e.g. \cite{P}), it is well known that if $\Ric\geq0$ then for every $x_0\in M$
  the function $\frac{\Vol(B_r(x_0))}{\omega_nr^n}$ is non-increasing in $r$, where $\omega_n$ is the volume of the Euclidean ball of radius $1$. In particular, for every $x_0\in M$ and every $r>0$
  \begin{equation}\label{volume}
  \Vol(B_r(x_0))\leq\omega_nr^n.
  \end{equation}

  Since by a celebrated result by Li-Yau (see e.g. \cite{LY}) if
  $\Vol(B_r(x_0))\leq Cr^2$ for some constant $C>0$ and every $r>0$ then $M$ is parabolic, if $\Ric\geq0$ and \eqref{uniformvolume} and \eqref{integrablef} hold, then $n=\operatorname{dim} M\geq3$.

  We also note that if $\Ric\geq0$ and condition \eqref{noncollapsing} holds, then by Bishop-Gromov volume comparison theorem we also have that for every $R\geq1$, every $0<r\leq R$ and every $x_0\in M$
  \begin{equation}\label{volume2}
  \Vol(B_r(x_0))\geq \Vol(B_R(x_0))\frac{r^n}{R^n}\geq \Vol(B_1(x_0))\frac{r^n}{R^n}\geq\frac{\alpha}{R^n}r^n.
  \end{equation}

  In particular, $\Vol(B_R(x_0))\geq\alpha$.

\begin{remark}\label{doubling}
  We recall that if $\Ric\geq0$ then geodesic balls are doubling with respect to the Riemannian measure on $M$, i.e. there exists $c_3>0$ such that for every $r>0$ and every $x_0\in M$ one has
  \begin{equation}\label{doub}
  \Vol(B_{2r}(x_0))\leq c_3\Vol(B_r(x_0)),
  \end{equation}
  see e.g. \cite{Gri3}.
\end{remark}

\begin{remark}
  Note that if conditions \eqref{uniformvolume} and \eqref{integrablef} hold on $M$ with $f(t)=t^{k-1}$ for some $k>1$ and if $\Ric\geq0$, then it must be $k\in(2,n]$. Indeed for such function by \eqref{integrablef} it must be $k>2$ , while by \eqref{uniformvolume} and \eqref{volume} we have
  $$
  \omega_nR^n\geq\Vol(B_R(x_0))\geq\frac{\Vol(B_{R_0}(x_0))}{\gamma R_0^{k}}R^{k}\qquad \textrm{for all }x_0\in M,\, R\geq R_0
  $$
  and hence $k\leq n$.

 Similarly, if conditions \eqref{uniformvolume} and \eqref{integrablef} hold on $M$ with $f(t)=t^{k-1}(\log t)^b$ for some $k>1$, $b\in\mathbb{R}$ and if $\Ric\geq0$, then it must be either $k\in(2,n)$, $b\in\mathbb{R}$ or $k=2$, $b>1$ or $k=n$, $b\leq0$.
\end{remark}

\subsection{On the space $L^1_G(M)$}\label{L1G}
Our goal here is to establish that, under our running assumptions, the \it strict \rm inclusion $L^1(M)\subsetneq L^1_G(M)$ holds, so that we are really considering a bigger space of initial data. We enforce this statement by also proving that any function behaving like $d(o,x)^{-a}$, with $a>2$, belongs to $L^1_G(M)$, provided volume of balls grows polynomially at infinity, \it independently of dimension\rm.

\begin{proposition}\label{Gprop}
 Let $M$ be a Riemannian manifold with $\Ric\geq0$ satisfying \eqref{noncollapsing}, \eqref{uniformvolume} and \eqref{integrablef}. Then:
 \begin{itemize}
 \item $L^1(M)\subsetneq L^1_G(M)$;
 \item Suppose, in addition, that there exist $\kappa_1>0, \kappa_2>0, \alpha\in (2, n]$ such that
\begin{equation}\label{e1m}
\kappa_1 R^{\alpha}\leq \Vol B_R(x_0)\leq \kappa_2 R^{\alpha} \quad \text{ for all }\; R\geq 1, x_0\in M\,.
\end{equation}
Let $o\in M$ be a fixed reference point, let $u:M\to [0, +\infty)$ be a measurable function. Suppose that, for some $a>0$,
\[u(x)\asymp \frac{1}{[1+ d(x, o)]^{a}} \quad \text{ for all }\; x\in M\,. \]
Then $u\in L^1(M)$ if and only if $a>\alpha$, whereas $u\in L^1_G(M)$ if and only if $a>2$.
\end{itemize}

\end{proposition}

\begin{proof}
To prove the first assertion, we shall construct an explicit function $v$ such that $v\in L^1_G(M)$ but not in $L^1(M)$. To this end, let $o\in M$ be fixed. Let $R_0$ be as in \eqref{uniformvolume}, and let us choose a sequence $\{o_j\}_{j\in \mathbb N}$ satisfying the following conditions:
\[\begin{aligned}
&d(o_1,o)\ge 4R_0+4,\ \ d(o_j,o)\ge 4d(o_{j-1},o),\,\forall j\ge2;\\
&\sum_{j=1}^{+\infty}\int_{d(o_j,o)-1}^{+\infty}\frac1{f(t)}\,{\rm d}t<+\infty.
\end{aligned}\]
Note that such a sequence indeed exists by condition \eqref{integrablef}. Let
\[
v:=\sum_{j=1}^{+\infty}{\mathbf 1}_{B_1(o_j)}.
\]
We claim that $v\in L^1_G(M)$ but $v\not\in L^1(M)$. The latter claim is immediate since
\[
\int_M v\,{\rm d}\mu=\sum_{j=1}^{+\infty} \textrm{Vol}\,(B_1(o_j))\ge \sum_{j=1}^{+\infty}\alpha=+\infty,
\]
where $\alpha$ is as in condition \eqref{noncollapsing}. The proof of the property $v\in L^1_G(M)$ is more involved. Hereafter, $C$ will denote any inessential positive constant, that may vary from line to line. We start noting that, for all $j>k\ge1$ we have
\begin{equation*}
\begin{aligned}
d(o_j,o_k)&\ge d(o_j,o)-d(o_k,o)\ge \left(4^{j-k}-1\right)d(o_k,o)\ge \left(4^{j-k}-1\right)4^{k-1}d(o_1,o)\\
&\ge 12(R_0+1).
\end{aligned}
\end{equation*}
We shall now estimate the following quantity, for $x_0\in M$, $j\in \mathbb N$:
\begin{equation}\label{18}
I:=\int_{M\setminus B_1(x_0)}\mathbf{1}_{B_1(o_j)}G(x,x_0)\,{\rm d}\mu(x)=\int_{\left\{M\setminus B_1(x_0)\right\}\cap B_1(o_j)}G(x,x_0)\,{\rm d}\mu(x)
\end{equation}

To this end, we shall distinguish two cases. The first one, that we name Case 1, occurs when $d(x_0,o_j)\ge R_0+2$ for all $j\in \mathbb N$. Otherwise, in Case 2 there exists a unique $k_0\ge1$ such that $d(x_0,o_{k_0})<R_0+2$ and $d(x_0,o_j)\ge R_0+2$ for all $j\not=k_0$.

We start discussing Case 1. In that case, by the bound \eqref{8} on the Green function, that will be proved later on in Lemma \ref{Greenest}:
\begin{equation}\label{I}
\begin{aligned}
I&\le  C\int_{\left\{M\setminus B_1(x_0)\right\}\cap B_1(o_j)}\left[ \frac{d(x,x_0)f(d(x,x_0))}{\Vol(B_{d(x,x_0)}(x_0))}\int_{d(x,x_0)}^{+\infty}\frac{1}{f(t)}\,{\rm d}t\right]\,{\rm d}\mu(x)\\
&\le C\int_{\left\{M\setminus B_1(x_0)\right\}\cap B_1(o_j)}\left[ \frac{[d(o_j,x_0)+1]f(d(o_j,x_0)+1)}{\Vol(B_{d(o_j,x_0)-1}(x_0))}\int_{d(o_j,x_0)-1}^{+\infty}\frac{1}{f(t)}\,{\rm d}t\right]\,{\rm d}\mu(x)\\
&\le C\left[ \frac{[d(o_j,x_0)+1]f(d(o_j,x_0)+1)}{\Vol(B_{d(o_j,x_0)-1}(x_0))}\int_{d(o_j,x_0)-1}^{+\infty}\frac{1}{f(t)}\,{\rm d}t\right]\textrm{Vol}\,\left(B_1(o_j)\right).
\end{aligned}\end{equation}
Observe that
\[
d(o_j,x_0)-1= \frac{d(o_j,x_0)+1}2+\frac{d(o_j,x_0)-1}2-1\ge  \frac{d(o_j,x_0)+1}2
\]
since $d(x_0,o_j)\ge R_0+2$ and $R_0\ge1$.  By \eqref{doub}
\[
\textrm{Vol}\left(\,B_{d(o_j,x_0)-1}(x_0)\right)\ge \textrm{Vol}\left(\,B_{\frac{d(o_j,x_0)+1}2}(x_0)\right)\ge C \textrm{Vol}\left(B_{d(o_j,x_0)+1}(x_0)\right).
\]
Then, using also \eqref{volume}, one gets from \eqref{I}:
\[
I\le C \frac{[d(o_j,x_0)+1]f(d(o_j,x_0)+1)}{\Vol(B_{d(o_j,x_0)+1}(x_0))}\int_{d(o_j,x_0)-1}^{+\infty}\frac{1}{f(t)}\,{\rm d}t
\]
Since $d(o_j,x_0)+1\ge R_0$, using \eqref{uniformvolume} we obtain
\begin{equation*}\label{I4}
I\le C \frac{R_0f(R_0)}{\Vol(B_{R_0}(x_0))}\int_{d(o_j,x_0)-1}^{+\infty}\frac{1}{f(t)}\,{\rm d}t\le C \int_{d(o_j,x_0)-1}^{+\infty}\frac{1}{f(t)}\,{\rm d}t
\end{equation*}
by the non-collapsing assumption \eqref{noncollapsing}.

Clearly, either $x_0\in \overline{B_{d(o_1,o)}(o)}$ or there exists a unique $j_0\in\mathbb N$ such that $x_0\in \overline{B_{d(o_{j_0+1},o)}(o)}\setminus B_{d(o_{j_0},o)}(o)$. In the first case,
\begin{equation}\label{I5}
\int_{d(x_0,o_1)-1}^{+\infty}\frac1{f(s)}\,{\rm d}s\le \int_{R_0}^{+\infty}\frac1{f(s)}\,{\rm d}s=\beta.
\end{equation}
by \eqref{integrablef}. On the other hand, if $j\ge2$
\[\begin{aligned}
d(x_0,o_j)&\ge d(o,o_j)-d(o,x_0)\ge d(o,o_j)-d(o,o_1)\\ &\ge 4 d(o,o_{j-1})-d(o_1,o)\ge d(o,o_{j-1})+3\cdot 4^{j-2}d(o_1,o)-d(o_1,o)\\
&\ge d(o,o_{j-1}).
\end{aligned}\]
Then
\[
\int_{d(x_0,o_j)-1}^{+\infty}\frac1{f(s)}\,\textrm{d}s\le \int_{d(o,o_{j-1})-1}^{+\infty}\frac1{f(s)}\,\textrm{d}s
\]
so that
\begin{equation*}\label{I2}
I\le C \int_{d(o,o_{j-1})-1}^{+\infty}\frac1{f(s)}\,\textrm{d}s,\ \ \ \forall j\ge2, \ \forall x_0\in \overline{B_{d(o_1,o)}(o)}.
\end{equation*}

\medskip\noindent We now consider the second case, namely $x_0\in \overline{B_{d(o_{j_0+1},o)}(o)}\setminus B_{d(o_{j_0},o)}(o)$ for a unique $j_0\in\mathbb N$. We shall prove first that
\begin{equation}\label{I3}
I\le C \int_{d(o,o_{j-1})-1}^{+\infty}\frac1{f(s)}\,\textrm{d}s,\ \ \ \forall x_0\in \overline{B_{d(o_{j_0+1},o)}(o)}\setminus B_{d(o_{j_0},o)}(o), \ \forall j\not=1, j\not=j_0+1.
\end{equation}
To this end, consider first the case $j\ge j_0+2$. One has:
\[
\begin{aligned}
d(x_0,o_j)&\ge d(o,o_j)-d(x_0,o)\ge d(o,o_j)-d(o,o_{j_0+1})\ge d(o,o_{j-1})+3d(o,o_{j-1})-d(o,o_{j_0+1})\\ &\ge d(o,o_{j-1}).
\end{aligned}
\]
Hence, for such $j$
\[
I\le C \int_{d(x_0,o_{j})-1}^{+\infty}\frac1{f(s)}\,\textrm{d}s\le C \int_{d(o,o_{j-1})-1}^{+\infty}\frac1{f(s)}\,\textrm{d}s.
\]
so that \eqref{I3} holds for such $j$. We now consider the case $2\le j\le j_0$. Since:
\[\begin{aligned}
d(x_0,o_j)&\ge d(x_0,o)-d(o,o_j)\ge d(o,o_{j_0})-d(o,o_j)\ge \left[4^{j_0-j}-1\right]d(o,o_j)\ge d(o,o_j)\\
&\ge d(o,o_{j-1}),
\end{aligned}\]
then \eqref{I3} holds also for $2\le j\le j_0$.
Finally, if $j=j_0+1$ or $j=1$, since $d(x_0,o_j)\ge R_0+2$, one has
\begin{equation*}\label{I6}
I\le C \int_{d(x_0,o_{j})-1}^{+\infty}\frac1{f(s)}\,\textrm{d}s\le C \int_{R_0}^{+\infty}\frac1{f(s)}\,\textrm{d}s\le C.
\end{equation*}

We are left with considering Case 2. In this second case, for all $j\neq k_0$ we have $d(x_0,o_j)\ge R_0+2$. Thus we can repeat the arguments of Case 1, in order to estimate $I$ in \eqref{18}, for every $j\neq k_0$. On the other hand, if $j=k_0$ we have
\[
d(x,x_0)\leq d(x,o_{k_0})+d(o_{k_0},x)\leq R_0+3,\ \ \ \forall x\in B_1(o_{k_0})
\]
so that by Lemma \ref{Greenest}-$iii)$ with $R=R_0+3$ we deduce
\begin{equation*}\label{I7}
I\le C\int_{\left\{M\setminus B_1(x_0)\right\}\cap B_1(o_{k_0})}\frac{1}{(d(x,x_0))^{n-2}}\,{\rm d}\mu(x)\le C\textrm{Vol}(B_1(o_{k_0}))\le C.
\end{equation*}

Finally, we note that, since $v\in L^\infty(M)$ with $\|v\|_{L^\infty(M)}=1$, for every $x_0\in M$ we have
\begin{equation}\label{I8}
\int_{B_1(x_0)}v(x)\,{\rm d}\mu(x)\le \textrm{Vol}(B_1(x_0))\le C.
\end{equation}

Then for each $x_0\in M$, by \eqref{I5}-\eqref{I8}, we have
\begin{equation*}
\begin{aligned}
\|v\|_{L^1_{x_0, G}}&=\int_{B_1(x_0)}v(x)\, d\mu(x) + \int_{M\setminus B_1(x_0)}v(x) G(x, x_0)\, d\mu(x)\\
&=\int_{B_1(x_0)}v(x)\, d\mu(x) + \sum_{j=1}^{+\infty}\int_{\left\{M\setminus B_1(x_0)\right\}\cap B_1(o_j)}G(x,x_0)\,{\rm d}\mu(x)\\
&\le C+C\sum_{j=1}^{+\infty}\int_{d(o_j,o)-1}^{+\infty}\frac1{f(t)}\,{\rm d}t\le C\,,
\end{aligned}
\end{equation*}
where $C>0$ is a suitable constant independent of $x_0\in M$. We conclude that, by definition, $v\in L^1_G(M)$. The proof of the first part of the statement is then complete.

As for the second part of the statement, to begin with we show that $u\in L^1(M)$ if and only if $a>\alpha.$ Assume first that $a>\alpha$. Let $\operatorname{Meas}(\partial B_r(x_0))$ denote the measure of the sphere centered at $x_0\in M$ and radius $r>0$. By means of the coarea formula and the fact that \begin{equation}\label{coarea}\operatorname{Meas}(\partial B_R(x_0))=\frac d{dr}\Vol(B_R(x_0)),\end{equation} integrating by parts yields
\begin{equation}\label{e8m}
\begin{aligned}
&\int_{B_R(o)} u(x) d\mu(x) \asymp \int_{B_R(o)}\frac{1}{[1+d(x, o)]^a} d\mu(x) =\int_0^R \frac{d}{dr} \Vol(B_r(o))\frac{1}{(1+r)^a} dr\\
&=\left.\frac{\Vol(B_r(o))}{(1+r)^a}\right|_0^R + a\int_0^R \Vol(B_r(o))\frac{1}{(1+r)^{a+1}}dr \asymp \frac{R^\alpha}{(1+R)^a}+1+\int_1^R\frac{r^\alpha}{(1+r)^{a+1}}dr\\
&\asymp \frac{R^\alpha}{(1+ R)^a}+ 1\,
\end{aligned}
\end{equation}
for every large enough $R\gg1$. Since
\[\int_M u(x) d\mu(x) = \lim_{R\to +\infty} \int_{B_R(o)} u(x) d\mu(x), \]
from \eqref{e8m} we deduce that $u\in L^1(M)$ if $a>\alpha\,.$  With similar arguments one easily shows that $u\not\in L^1(M)$ if $a\le\alpha$. 
Now, we show that $u\in L^1_G(M)$ whenever $a>2$.

Let $x_0\in M$. Observe that, in view of \eqref{Green},
\begin{equation}\label{e2m}
\begin{aligned}
\int_{M\setminus B_1(x_0)}& u(x)\, G(x_0, x) d\mu \\ &\leq c_2 \int_{B^c_1(x_0)\cap B_1(o)} u(x) \int_{d(x, x_0)}^{+\infty}\frac{t }{\Vol(B_t(x_0))}dt d\mu(x)\\& +
c_2\int_{B^c_1(x_0)\cap B^c_1(o)} u(x) \int_{d(x_0, x)}^{+\infty}\frac{t}{\Vol(B_t(x_0))}dt d\mu(x)\\
&\leq c_2\Vol(B_1(o))\|u\|_\infty \int_{1}^{+\infty}\frac{t }{\Vol(B_t(x_0))}dt\\ & + C\frac{c_2}{\kappa_1} \int_{B_1^c(x_0)}\frac 1{[1+ d(x, o)]^a}\int_{d(x_0, x)}^{+\infty}t^{1-\alpha}dt d\mu(x)\\
& \leq C + C\int_{M}\frac{1}{[d(x, o)+1]^a}[d(x, x_0)+1]^{2-\alpha} d\mu(x)\,.
\end{aligned}
\end{equation}
Now in order to estimate the last integral, we distinguish three cases: $x\in B_{\frac{d(x_0, o)}{2}}(o), x\in B^c_{2 d(x_0, o)}(o)$ and $x\in B_{2 d(x_0, o)}(o)\setminus B_{\frac{d(x_0, o)}{2}}(o).$

Let $x\in B_{\frac{d(x_0, o)}{2}}(o).$ Then $d(x, x_0)\geq \frac{d(x_0, o)}{2}\,.$ Therefore, in view of the coarea formula and \eqref{e1m}, arguing as in \eqref{e8m},
\begin{equation*}\label{e3m}
\begin{aligned}
&\int_{B_{\frac{d(x_0, o)}2}(o)}\frac{1}{[d(x, o) +1]^a}\frac{1}{[d(x_0, x)+1]^{\alpha-2}}d\mu(x)\\ &\leq \frac{C}{[d(x_0,o)]^{\alpha-2}}\int_{B_{\frac{d(x_0, o)}2}(o)}\frac 1{[d(x, o) +1]^a} d\mu(x)\\
& \leq C \frac{[1+d(x_0, o)]^{\alpha-a}}{[1+d(x_0, o)]^{\alpha-2}} + C \leq \frac{C}{[d(x_0,o)+1]^{a-2}}+ C \leq C\,,
\end{aligned}
\end{equation*}
provided that $a\geq 2$.
Let $x\in B^c_{2 d(x_0, o)}(o)$. Thus $d(x, x_0)\geq \frac{d(x, o)}{2}$. Hence, again by the coarea formula and \eqref{e1m}, arguing as in \eqref{e8m},
\begin{equation*}\label{e4m}
\begin{aligned}
&\int_{B^c_{2 d(x_0, o)}(o)}\frac1{[d(x, o)+1]^a}\frac 1{[d(x_0, x)+1]^{\alpha-2}}d\mu(x)\\ &
\leq \int_{B^c_{2 d(x_0, o)}(o)}\frac1{[d(x, o)+1]^{a+\alpha-2}}d\mu(x) \leq C,
\end{aligned}
\end{equation*}
provided that $a>2$.

Finally, let $x\in B_{2 d(x_0, o)}(o)\setminus B_{\frac{d(x_0, o)}{2}}(o).$ Then
\begin{equation}\label{e5m}
\begin{aligned}
&\int_{ B_{2 d(x_0, o)}(o)\setminus B_{\frac{d(x_o, o)}{2}}(o)}\frac1{[d(x, o)+1]^a}\frac 1{[d(x_0, x)+1]^{\alpha-2}}d\mu(x)\\&
\leq \frac{C}{[d(x_0, o)+1]^a}\int_{B_{3d(x_0, o)}(x_0)}\frac1{[d(x_0, x)+1]^{\alpha-2}}d\mu(x)\leq \frac{C d(x_0, o)^2}{[d(x_0, o)+1]^a}+ C \leq C\,,
\end{aligned}
\end{equation}
provided that $a\geq 2$. Putting together \eqref{e2m}-\eqref{e5m}, we can infer that if $a>2$, then
\begin{equation}\label{e6m}
\sup_{x_0\in M} \int_{M\setminus B_1(x_0)} u(x)\, G(x_0, x) d\mu(x) \leq C\,.
\end{equation}
Furthermore, by virtue of \eqref{e1m},
\begin{equation}\label{e7m}
\sup_{x_0\in M} \int_{B_1(x_0)} u(x) d\mu(x) \leq \|u\|_\infty \kappa_2\,.
\end{equation}
From \eqref{e6m} and \eqref{e7m} we deduce that $u\in L^1_G(M)$ whenever $a>2.$ On the other hand, using similar arguments, it is easily seen that
\[\int_{M\setminus B_1(o)}u(x) G(o, x)\, d\mu(x) = +\infty, \]
whenever $0<a\leq 2$. Thus $u\not\in L^1_{o, G}(M)$ and hence $u\not\in L^1_G(M)$.
\end{proof}

\subsection{Examples}\label{ex} We provide here a number of examples in which conditions \eqref{noncollapsing}, \eqref{uniformvolume} and \eqref{integrablef} hold.

\begin{example}
  Let $(M,g)$ be a Riemannian manifold with $\Ric\geq0$ such that
  $$
  C_1 R^k(\log R)^b\leq\Vol(B_R(x))\leq C_2R^k(\log R)^b \qquad\textrm{for all }x\in M,\,R\geq R_0
  $$
  for some $C_1,C_2>0$, $R_0>1$, with $k\in(2,n)$, $b\in\mathbb{R}$ or $k=2$, $b>1$ or $k=n$, $b\leq0$. Then \eqref{uniformvolume} and \eqref{integrablef} hold, with
  $$f(R)=R^{k-1}(\log R)^b,\qquad R\geq R_0.$$
  Indeed $f$ trivially satisfies \eqref{integrablef} and for every $R\geq r\geq R_0$, $x\in M$ we have
  \begin{equation*}
    \frac{Rf(R)}{\Vol(B_R(x))}=\frac{R^k(\log R)^b}{\Vol(B_R(x))}\leq \frac{1}{C_1}\leq \frac{C_2}{C_1}\frac{r^k(\log r)^b}{\Vol(B_r(x))}=\gamma \frac{rf(r)}{\Vol(B_r(x))}
  \end{equation*}
  that is \eqref{uniformvolume}, with $\gamma=\frac{C_2}{C_1}$.
  \end{example}

\begin{example}
  Let $(M,g)$ be a Riemannian manifold with $n=\operatorname{dim} M\geq3$, $\Ric\geq0$ and maximal volume growth, i.e.
  $$
  A:=\lim_{R\rightarrow+\infty}\frac{\Vol(B_R(x))}{\omega_nR^n}=AVR(M)>0
  $$
  for some, and hence for all, $x\in M$. $A$ is called asymptotic volume ratio of $M$. Then for every $x\in M$, $R>0$ by Bishop-Gromov volume comparison theorem
  $$
  A\omega_nR^n\leq\Vol(B_R(x))\leq\omega_nR^n
  $$
  and therefore \eqref{uniformvolume} holds with $f(R)=R^{n-1}$, which clearly satisfies \eqref{integrablef}, see the previous example.
\end{example}

\begin{example}
  Let $M$ be a homogeneous manifold, i.e. for every $x,y\in M$ there exists an isometry $\varphi:M\rightarrow M$ such that $\varphi(x)=y$. Assume that $\Ric\geq0$ and $M$ non-parabolic. Then $M$ satisfies \eqref{noncollapsing}, \eqref{uniformvolume} and \eqref{integrablef}, with $f(R)=\frac{\Vol(B_R(x))}{R}$ for any $x\in M$.

  Indeed for every $x,y\in M$ and $R>0$ we have $\varphi(B_R(x))=B_R(y)$ for some isometry $\varphi:M\rightarrow M$, and thus $\Vol(B_R(x))=\Vol(B_R(y))$. In particular, choosing $R=1$ we see that \eqref{noncollapsing} is satisfied. If $f(R)=\frac{\Vol(B_R(x))}{R}$ for some fixed $x\in M$, then \eqref{integrablef} holds since
  $$
  \int_{1}^{+\infty}\frac{1}{f(t)}\,{\rm d}t=\int_{1}^{+\infty}\frac{t}{\Vol(B_t(x))}\,{\rm d}t
  $$
  and $M$ is non-parabolic by assumption, and thus the integral is finite. Lastly we see that \eqref{uniformvolume} holds since for every $y\in M$ and every $R>0$
  $$
  \frac{Rf(R)}{\Vol(B_R(y))}=\frac{\Vol(B_R(x))}{\Vol(B_R(y))}=1.
  $$

  Using similar arguments we see that $M$ satisfies \eqref{noncollapsing}, \eqref{uniformvolume} and \eqref{integrablef}, with $f(R)=\frac{\Vol(B_R(x))}{R}$ for a fixed $x\in M$, if $M$ is non-parabolic, $\Ric\geq0$ and there exist constants $R_0\geq1$, $C_1,C_2>0$ such that for every $x,y\in M$, $R\geq R_0$
  $$
  C_1\Vol(B_R(x))\leq\Vol(B_R(y))\leq C_2\Vol(B_R(x)).
  $$
\end{example}

\begin{example}
  Let $M$ be a Riemannian manifold with $\Ric\geq0$ such that for some $\varepsilon>0$ and $R_0\geq1$ one has
  \begin{equation}\label{e1d}
  \Vol(B_{2R}(x))\geq 2^{2+\varepsilon}\Vol(B_R(x))\qquad \textrm{for every }x\in M,\,R\geq R_0.
  \end{equation}
  Then \eqref{uniformvolume}, \eqref{integrablef} are satisfied, with $f(R)=R^{1+\varepsilon}$.

  Indeed $f$ trivially satisfies \eqref{integrablef}. Moreover, if $x\in M$, $R\geq r\geq R_0$ let $k\in\mathbb{N}$ be such that $2^kr\leq R<2^{k+1}r$. Then by \eqref{e1d}
  $$
  \Vol(B_R(x))\geq\Vol(B_{2^kr}(x))\geq 2^{(2+\varepsilon)k}\Vol(B_r(x))>\frac{1}{2^{2+\varepsilon}}\left(\frac{R}{r}\right)^{2+\varepsilon}\Vol(B_r(x)).
  $$
  Thus
  $$
  \frac{Rf(R)}{\Vol(B_R(x))}=\frac{R^{2+\varepsilon}}{\Vol(B_R(x))}<2^{2+\varepsilon}\frac{r^{2+\varepsilon}}{\Vol(B_r(x))}=2^{2+\varepsilon}\frac{rf(r)}{\Vol(B_r(x))},
  $$
  that is \eqref{uniformvolume}.
\end{example}

\begin{example}
  Let $M$ be a Riemannian manifold with $\Ric\geq0$. Assume there exists $\gamma_1,\gamma_2>0$, $R_0\geq1$,  a nondecreasing function $f:[R_0,+\infty)\rightarrow(0,+\infty)$ and $h:M\times[R_0,+\infty)\rightarrow(0,+\infty)$ such that for every $x\in M$ the function $h(x,\cdot)$ is non-decreasing and
  $$
  \gamma_1h(x,R)\leq\frac{\Vol(B_R(x))}{Rf(R)}\leq\gamma_2h(x,R)\qquad\textrm{for every }R\geq R_0.
  $$
  Then condition \eqref{uniformvolume} holds on $M$.

  Indeed for every $x\in M$ and every $R\geq r\geq R_0$ we have
  $$
  \frac{Rf(R)}{\Vol(B_R(x))}\leq\frac{1}{\gamma_1h(x,R)}\leq\frac{1}{\gamma_1h(x,r)}\leq \frac{\gamma_2}{\gamma_1}\frac{rf(r)}{\Vol(B_r(x))}.
  $$
\end{example}

\begin{example}\label{exproduct}
  Let $X$ be a Riemannian manifold with $\Ric_X\geq0$ satisfying \eqref{noncollapsing}, \eqref{uniformvolume}, \eqref{integrablef}, and let $N$ be a compact Riemannian manifold such that $\Ric_N\geq 0$. Then $M=X\times N$, endowed with the product metric, also satisfies $\Ric_M\geq0$ and \eqref{noncollapsing}, \eqref{uniformvolume}, \eqref{integrablef}, for the same function $f$ as that of $X$.

  Indeed, it is well known that $\Ric_M\geq0$. Moreover, if we denote by $d_X,d_N,d_M$ the geodesic distances on each of the three manifolds and by
  $$
  d_2((x,p),(y,q))=\sqrt{d_X^2(x,y)+d_N^2(p,q)}\qquad \textrm{for every }x,y\in X,\, p,q\in N,
  $$
  then there exists $C>0$ such that
  $$
  \frac{1}{C}d_M((x,p),(y,q))\leq d_2((x,p),(y,q))\leq C d_M((x,p),(y,q))\quad \textrm{for every }x,y\in X,\, p,q\in N.
  $$
  If $\rho_0$ denotes the diameter of $N$, then for every $R\geq\rho_0$ it is easy to see that
  $$
  B^M_R((x,p))\subset B^X_{CR}(x)\times N,\qquad B_R^X(x)\times N\subset B^M_{\sqrt2CR}((x,p))
 $$
 for every $x\in X$, $p\in N$. Recalling the doubling property of geodesic balls in manifolds with nonnegative Ricci curvature, see Remark \ref{doubling}, we thus obtain for every $x\in X$, $p\in N$
 $$
 \Vol_M(B^M_R((x,p)))\leq \Vol_X(B^X_{CR}(x))\Vol_N(N)\leq C_2\Vol_X(B^X_R(x))
 $$
 for some $C_2>0$. Similarly,
 $$
 C_1\Vol_X(B^X_{R}(x))\leq \Vol_M(B^M_{R}((x,p)))
 $$
 for every $x\in X$, $p\in N$, for some $C_1>0.$ Now if $\rho_0\leq1$ we readily see that, since $X$ is non collapsing,
 $$
 \inf_{(x,p)\in M}\Vol_M(B^M_{1}((x,p)))\geq C_1\inf_{x\in X}\Vol_X(B^X_{1}(x))>0.
 $$
 On the other hand, if $\rho_0>1$,
 $$
 \inf_{(x,p)\in M}\Vol_M(B^M_{\rho_0}((x,p)))\geq C_1\inf_{x\in X}\Vol_X(B^X_{\rho_0}(x))>C_1\inf_{x\in X}\Vol_X(B^X_{1}(x))>0.
 $$
 Since $\Ric_M\geq0$, by the doubling property of geodesics balls in $M$ we conclude that
 $$
 \inf_{(x,p)\in M}\Vol_M(B^M_{1}((x,p)))\geq c_3^{-k} \inf_{(x,p)\in M}\Vol_M(B^M_{\rho_0}((x,p)))>0,
 $$
 where $k$ is the smallest integer such that $2^k\geq\rho_0.$ Thus \eqref{noncollapsing} holds on $M$.

 If $f$ is the function satisfying \eqref{uniformvolume}, \eqref{integrablef} on $X$, for every $(x,p)\in M$ and every $R\geq r\geq \max\{R_0,\rho_0\}$ we have
 $$
 \frac{Rf(R)}{\Vol_M(B^M_R((x,p)))}\leq \frac{1}{C_1}\frac{Rf(R)}{\Vol_X(B^X_R(x))} \leq \frac{\gamma }{C_1}\frac{rf(r)}{\Vol_X(B^X_r(x))} \leq \frac{\gamma C_2 }{C_1} \frac{rf(r)}{\Vol_M(B^M_r((x,p)))}.
 $$
 That is, conditions \eqref{uniformvolume}, \eqref{integrablef} hold on $M$, for the same function $f$ as that of $X$.
  \end{example}

\section{A priori estimates}\label{apriori}

We prove in this section a number of technical results, that will be crucial in proving our main ones later. We start with some estimates on the Green function.

\begin{lemma}\label{Greenest}
  Let $M$ be a $n$-dimensional Riemannian manifold with $\Ric\geq0$ and assume that conditions \eqref{noncollapsing}, \eqref{uniformvolume} and \eqref{integrablef} hold. Let $G$ be the minimal positive Green function on $M$. Then
  \begin{itemize}
    \item[i)] for every $x,x_0\in M$ $$G(x,x_0)\geq \frac{c_1}{(n-2)\omega_n}\frac{1}{(d(x,x_0))^{n-2}},$$
    \item[ii)] for every $x,x_0\in M$ with $d(x,x_0)\geq R_0$
     \begin{equation}\label{8}
   G(x,x_0)\leq c_2\gamma \frac{d(x,x_0)f(d(x,x_0))}{\Vol(B_{d(x,x_0)}(x_0))}\int_{d(x,x_0)}^{+\infty}\frac{1}{f(t)}\,{\rm d}t,
   \end{equation}
    \item[iii)] for every $x,x_0\in M$ with $d(x,x_0)\leq R$, with $R\ge R_0$,
    \begin{equation*}
    G(x,x_0)\leq \frac{c_2}{\alpha}\left[\frac{R^n}{n-2}+\gamma\beta f(R)R^{n-1}\right]\frac{1}{(d(x,x_0))^{n-2}},
    \end{equation*}
    \item[iv)] for every $0<R<R_0$ and every $x_0\in M$
    \begin{equation}\label{14}
    \int_{B_R(x_0)}G(x,x_0)\,{\rm d}\mu(x)\leq \frac{\omega_n c_2 n}{2\alpha}\left[\frac{R_0^n}{n-2}+\gamma\beta f(R_0)R_0^{n-1}\right]R^2
    \end{equation}
    \item[v)] for every $R\geq R_0$ and every $x_0\in M$
    \begin{equation}\label{5}\begin{aligned}
    \int_{B_R(x_0)}G(x,x_0)\,{\rm d}\mu(x)&\leq c_2\max\{\gamma,\tfrac{1}{2}\} \left(Rf(R)\int_R^{+\infty}\frac{1}{f(t)}\,{\rm d}t+R^2\right)\\
    &=: c_2\max\{\gamma,\tfrac{1}{2}\} h(R).
    \end{aligned}\end{equation}
  \end{itemize}
\end{lemma}

\begin{remark}\label{15}
  Note that the function $h$ defined in \eqref{5} is increasing and diverging at $+\infty$ as $R$ tends to $+\infty$. Indeed,
  $$
  h'(R)=f(R)\int_R^{+\infty}\frac{1}{f(t)}\,{\rm d}t+Rf'(R)\int_R^{+\infty}\frac{1}{f(t)}\,{\rm d}t+R\geq R_0\geq1,\qquad R\geq R_0
  $$
  since $f$ is positive and nondecreasing.

  Note also that by Lemma \ref{Greenest}-ii), \eqref{noncollapsing}, \eqref{uniformvolume} and \eqref{integrablef} for every $x,x_0\in M$ with $d(x,x_0)\geq R_0$ we have
  $$
  G(x,x_0)\leq c_2\gamma \frac{d(x,x_0)f(d(x,x_0))}{\Vol(B_{d(x,x_0)}(x_0))}\int_{d(x,x_0)}^{+\infty}\frac{1}{f(t)}\,{\rm d}t\leq \frac{c_2\gamma\beta}{\alpha}R_0f(R_0).
  $$
  while by Lemma \ref{Greenest}-iii) for every $x,x_0\in M$ with $1\leq d(x,x_0)\leq R_0$ we have
  $$
  G(x,x_0)\leq \frac{c_2}{\alpha}\left[\frac{R_0^n}{n-2}+\gamma\beta f(R_0)R_0^{n-1}\right]\,.
  $$
\end{remark}

\begin{example}
  If $f(t)=t^{k-1}(\log t)^b$ for some $k\in(2,n)$, $b\in\mathbb{R}$ or $k=n$, $b\leq0$, then for some constants $a_1,a_2>0$

$$a_1R^2\leq h(R)\leq a_2R^2\qquad\textrm{for every }R\geq R_0.$$

If $f(t)=t(\log t)^b$ for some $b>1$, then $$h(R)=R^2\left(\frac{1}{b-1}\log R+1\right)\qquad\textrm{for every }R\geq R_0.$$

\end{example}

\begin{proof}[Proof of Lemma \ref{Greenest}]
  By our assumptions $M$ is non-parabolic, see Remark \ref{rem1}. By \eqref{Green} and \eqref{volume} we have $$G(x,x_0)\geq c_1\int_{d(x,x_0)}^{+\infty}\frac{t}{\Vol(B_t(x_0))}\,{\rm d}t\geq \frac{c_1}{\omega_n}\int_{d(x,x_0)}^{+\infty}\frac{1}{t^{n-1}}\,{\rm d}t,$$
  which gives $i)$.

  By \eqref{Green} and \eqref{uniformvolume}, if $d(x,x_0)\geq R_0$,
  \begin{align*}
  G(x,x_0)&\leq c_2\int_{d(x,x_0)}^{+\infty}\frac{t}{\Vol(B_t(x_0))}\,{\rm d}t\\
  &=c_2\int_{d(x,x_0)}^{+\infty}\frac{tf(t)}{\Vol(B_t(x_0))}\frac{1}{f(t)}\,{\rm d}t\\
  &\leq c_2\gamma\frac{d(x,x_0)f(d(x,x_0))}{\Vol(B_{d(x,x_0)}(x_0))}\int_{d(x,x_0)}^{+\infty}\frac{1}{f(t)}\,{\rm d}t,
  \end{align*}
  that is $ii)$.

 If $d(x,x_0)\leq R$, with $R\geq R_0$, by \eqref{Green}
  $$
  G(x,x_0)\leq c_2\int_{d(x,x_0)}^{+\infty}\frac{t}{\Vol(B_t(x_0))}\,{\rm d}t=c_2\int_{d(x,x_0)}^{R}\frac{t}{\Vol(B_t(x_0))}\,{\rm d}t+c_2\int_{R}^{+\infty}\frac{t}{\Vol(B_t(x_0))}\,{\rm d}t.
  $$
  Then by \eqref{volume2}, \eqref{uniformvolume} and \eqref{integrablef}
  \begin{align*}
    G(x,x_0) & \leq c_2\frac{R^n}{\alpha} \int_{d(x,x_0)}^{R}\frac{1}{t^{n-1}}\,{\rm d}t+c_2\gamma\beta\frac{Rf(R)}{\Vol(B_{R}(x_0))}\\
    & \leq \frac{c_2R^n}{(n-2)\alpha}\frac{1}{(d(x,x_0))^{n-2}}+c_2\gamma\beta\frac{Rf(R)}{\alpha}\\
    &\leq\frac{c_2}{\alpha}\left[\frac{R^n}{n-2}+\gamma\beta f(R)R^{n-1}\right]\frac{1}{(d(x,x_0))^{n-2}},
  \end{align*}
  that is $iii)$.

   By using $iii)$ with $R=R_0$ and \eqref{volume}, for every $0<r<R_0$ and $x_0\in M$, we have
   \[\int_{B_r(x_0)}G(x, x_0) d\mu(x) \leq \frac{c_2}{\alpha}\left[\frac{R_0^n}{n-2}+\gamma\beta f(R_0)R_0^{n-1}\right]\int_{B_r(x_0)}\frac{1}{(d(x,x_0))^{n-2}}d\mu(x)\,.\]
   In view of the coarea formula, recalling \eqref{coarea}, integrating by parts yields $iv).$

  Now let $R\geq R_0$.  By \eqref{coarea}, the coarea formula and \eqref{Green}, integrating by parts we obtain
  \begin{align*}
    \int_{B_R(x_0)}G(x,x_0)\,{\rm d}\mu(x) & \leq c_2 \int_{B_R(x_0)}\int_{d(x,x_0)}^{+\infty}\frac{t}{\Vol(B_t(x_0))}\,{\rm d}td\mu(x)\\
    & = c_2 \int_0^R \operatorname{Meas}(\partial B_r(x_0))\int_r^{+\infty} \frac{t}{\Vol(B_t(x_0))}\,{\rm d}tdr\\
    &= c_2 \Vol(B_R(x_0))\int_R^{+\infty}\frac{t}{\Vol(B_t(x_0))}\,{\rm d}t + \frac{c_2}{2}R^2\\
    &\leq c_2\gamma Rf(R)\int_R^{+\infty}\frac{1}{f(t)}\,{\rm d}t+\frac{c_2}{2}R^2\\
    &\leq c_2\max\{\gamma,\tfrac{1}{2}\} \left(Rf(R)\int_R^{+\infty}\frac{1}{f(t)}\,{\rm d}t+R^2\right),
  \end{align*}
  which is $v)$.
\end{proof}

We now prove a crucial bound, that can be qualitatively stated by saying that the potential of a test function is quantitatively bounded above and below by the Green function, away from the pole. A similar bound was proved in \cite{BBGM}, by a completely different method involving the parabolic Harnack inequality. In the present geometrical setting, we provide a proved based only on elliptic estimates.
\begin{proposition}\label{stimapr}
  Let $M$ be a $n$-dimensional Riemannian manifold with $\Ric\geq0$ and assume that conditions \eqref{noncollapsing}, \eqref{uniformvolume} and \eqref{integrablef} hold.
  Let $\psi\in C(M), \psi\geq 0$ with $\operatorname{supp}\, \psi \subset B_{\sigma}(x_0)$, for some $x_0\in M$ and $\sigma\in (0, 1]$. Then there exist two constants $\gamma_1>0$ and $\gamma_2>0$, only depending on $c_1, c_2, c_3, \alpha, \beta, \gamma, n, R_0$, such that
 \begin{equation}\label{e1f}
 \gamma_1\|\psi\|_1 \left(d(x, x_0)^{n-2}\wedge 1\right) G(x, x_0) \leq (-\Delta)^{-1}\psi(x)\leq \gamma_2 \|\psi\|_\infty \operatorname{Vol}(B_\sigma(x_0))G(x, x_0)
 \end{equation}
for any $x\in M\,.$
\end{proposition}
\begin{proof}
First we show the estimate from above in \eqref{e1f}. In order to do it, first we suppose that $d(x, x_0)\geq 2\sigma.$ Then
\[d(x, x_0)-\sigma\leq d(x, x_0)- d(x_0, y)\leq d(x,y)\quad \text{ for all }\, y\in B_\sigma(x_0)\,.\]
Therefore, by using \eqref{Green},
 \begin{equation}\label{e3f}
 \begin{aligned}
 (-\Delta)^{-1}\psi(x)&\leq c_2 \int_{B_{\sigma}(x_0)}\psi(y)\int_{d(x,x_0)-\sigma}^\infty \frac{t}{\Vol(B_t(x))} dt d\mu(y)\\
 &\leq c_2\|\psi\|_\infty \Vol(B_\sigma(x_0)) \int_{d(x,x_0)-\sigma}^\infty \frac{t}{\Vol(B_t(x))} dt\,,
 \end{aligned}
 \end{equation}
 for some $c_2>0$. We now perform the change of variable $\tau:= t+\sigma$. Thus \eqref{e3f} yields
  \begin{equation}\label{e4f}
 \begin{aligned}
 (-\Delta)^{-1}\psi(x)&\leq c_2\|\psi\|_\infty \Vol(B_\sigma(x_0)) \int_{d(x,x_0)}^\infty \frac{\tau-\sigma}{\Vol(B_{\tau-\sigma}(x))} d\tau\\
& \leq c_2 \|\psi\|_\infty \Vol(B_\sigma(x_0))  \int_{d(x,x_0)}^\infty \frac{\tau}{\Vol(B_{\tau-\sigma}(x))} d\tau\,.
 \end{aligned}
 \end{equation}
Observe that $\tau\geq d(x, x_0)\geq 2\sigma$ implies $\tau-\sigma\geq \frac{\tau}2\,.$ Hence, by the doubling property (see Remark \ref{doubling}),
\[\Vol(B_{\tau-\sigma}(x))\geq \Vol(B_{\tau/2}(x))\geq c_3^{-1} \Vol(B_{\tau}(x)),\]
for some $c_3>0$. Consequently, from \eqref{e4f} we can infer that
\begin{equation}\label{e5f}
\begin{aligned}
(-\Delta)^{-1}\psi(x)&\leq c_3 c_2 \|\psi\|_\infty \Vol(B_{\sigma}(x_0))\int_{d(x,x_0)}^\infty \frac{\tau}{\Vol(B_{\tau}(x))} d\tau\\
&\leq \frac{c_3 c_2}{c_1} \|\psi\|_\infty \Vol(B_{\sigma}(x_0)) G(x, x_0)\,,
\end{aligned}
\end{equation}
which implies the estimate from above in \eqref{e1f}, whenever $d(x, x_0)\geq 2\sigma.$ Now instead assume that $0<d(x,x_0)<2\sigma.$ We have
\begin{equation}\label{e6f}
\begin{aligned}
(-\Delta)^{-1}\psi(x) &= \int_{B_\sigma(x_0)}\psi(x) G(x, y) d\mu(y) \\
&\leq \|\psi\|_\infty \int_{B_{\sigma}(x_0)}G(x,y) d\mu(y)\,.
\end{aligned}
\end{equation}
Observe that  if $y\in B_\sigma(x_0)$, then
$$
d(x,y)\leq d(x, x_0)+d(x_0, y)\leq 3\sigma\,.
$$
Thus
\begin{equation*}\label{e7f}
\begin{aligned}
\int_{B_\sigma(x_0)}G(x,y) d\mu(y)& \leq \int_{B_{3\sigma}(x)} G(x,y) d\mu(y) \\
&\leq c_2 \int_{B_{3\sigma}(x)}\int_{d(x,y)}^{\infty}\frac{t}{\Vol(B_{t}(x))}dt d\mu(y)\\ &= c_2 \int_0^{3\sigma}\operatorname{Meas}(\partial B_r(x))\int_{r}^{\infty}\frac{t}{\Vol(B_{t}(x))} dt dr,
\end{aligned}
\end{equation*}
here the coarea formula has been used. Integrating by parts, we get
\begin{equation}\label{e8f}
\begin{aligned}
\int_{B_\sigma(x_0)} G(x,y) d\mu(y) &\leq c_2 \left[\Vol(B_r(x))\int_r^{\infty}\frac{t}{\Vol(B_{t}(x))} dt \right]_{r=0}^{r=3\sigma}+ c_2 \int_0^{3\sigma}r dr \\
&= c_2 \Vol(B_{3\sigma}(x))\int_{3\sigma}^{\infty}\frac{t}{\Vol(B_{t}(x))} dt + \frac{9 c_2}{2}\sigma^2\,.
\end{aligned}
\end{equation}
Note that, for any $y\in B_{3\sigma}(x)$,
\[d(y, x_0)\leq d(x,y)+ d(x, x_0) \leq 3\sigma + 2\sigma < 8\sigma\,,\]
 so
 \[B_{3\sigma}(x)\subset B_{8\sigma}(x_0)\,.\]
 By the doubling property (see Remark \ref{doubling}),
 \begin{equation}\label{e9f}
 \Vol(B_{3\sigma}(x))\leq \Vol(B_{8\sigma}(x_0))\leq c_3^3 \Vol(B_\sigma(x_0))\,.
 \end{equation}
By \eqref{volume2}, since $\sigma\in(0,1]$, we have
\begin{equation}\label{e11f}
\Vol(B_\sigma(x_0))\geq \alpha\sigma^n.
\end{equation}
Thus, from Lemma \ref{Greenest}-(i) and \eqref{e11f} we obtain
 \begin{equation}\label{e12bf}
 \sigma^2=\frac{\sigma^n}{\sigma^{n-2}}\leq \frac{2^{n-2}}{\alpha}\Vol(B_\sigma(x_0))\frac 1{d(x, x_0)^{n-2}}\leq \frac{2^{n-2}(n-2)\omega_n}{\alpha c_1}\Vol(B_\sigma(x_0))G(x, x_0)\,.
 \end{equation}
 From \eqref{Green}, \eqref{e6f}, \eqref{e8f}, \eqref{e9f} and \eqref{e12bf}, we can infer that
 \begin{equation}\label{e12f}
 \begin{aligned}
 (-\Delta)^{-1}\psi(x)&\leq \|\psi\|_\infty \int_{B_{\sigma}(x_0)} G(x,y) d\mu(y)\\
& \leq c_2 c_3^3\|\psi\|_\infty \Vol(B_\sigma(x_0))\int_{d(x, x_0)}^{\infty}\frac{t}{\Vol(B_t(x))}dt\\
& + \frac{9(n-2)2^{n-3}c_2\omega_n}{\alpha c_1}\|\psi\|_\infty \Vol(B_\sigma(x_0))G(x, x_0)\\ &\leq \tilde \gamma_2\|\psi\|_\infty \Vol(B_\sigma(x_0))G(x, x_0),
 \end{aligned}
 \end{equation}
 with
 \[\tilde \gamma_2=\frac{c_2 c_3^3}{c_1}+\frac{9(n-2) 2^{n-3}c_2\omega_n}{\alpha c_1},\]
that yields the estimate from above in \eqref{e1f}, whenever $0<d(x,x_0)<2\sigma.$
From \eqref{e5f} and \eqref{e12f} we obtain the estimate from above in \eqref{e1f} for any $x\in M$, with $\gamma_2=\max\left\{\frac{c_3 c_2}{c_1}, \tilde\gamma_2\right\}=\tilde \gamma_2\,.$

Let us now show the inequality from below in \eqref{e1f}. To this end, note that in view of \eqref{Green}, for any $x\in M$,
\begin{equation}\label{e13f}
\begin{aligned}
(-\Delta)^{-1}\psi(x) &= \int_{B_\sigma(x_0)}\psi(y) G(x,y) d\mu(y) \\
&\geq c_1 \int_{B_\sigma(x_0)}\psi(y) \int_{d(x,y)}^{\infty} \frac{t}{\Vol(B_t(x))}dt d\mu(y)\,
\end{aligned}
\end{equation}
for some $c_1>0$. Since, for any $y\in B_\sigma(x_0)$,
\[d(x,y)\leq d(x, x_0)+ d(y, x_0)\leq d(x, x_0)+\sigma,\]
from \eqref{e13f}, by performing the change of variable $\tau:=t-\sigma,$ we obtain
\begin{equation}\label{e14f}
\begin{aligned}
(-\Delta)^{-1}\psi(x) &\geq c_1  \int_{B_\sigma(x_0)}\psi(y) \int_{d(x,x_0)+\sigma}^{\infty} \frac{t}{\Vol(B_t(x))}dt d\mu(y)\\
&\geq c_1\|\psi\|_1  \int_{d(x,x_0)}^{\infty} \frac{\tau +\sigma}{\Vol(B_{\tau+\sigma}(x))}d\tau.
\end{aligned}
\end{equation}
We now distinguish two cases: $(a)$ $d(x, x_0)\geq 1$ and $(b)$ $d(x, x_0)<1.$
We begin with case $(a)$. Since
$\tau\geq d(x, x_0)\geq 1\geq \sigma,$
we have $\tau+\sigma\leq 2\tau.$ By the doubling property (see Remark \ref{doubling}),
\[\Vol(B_{\tau+\sigma}(x_0))\leq \Vol(B_{2\tau(x_0)})\leq c_3 \Vol(B_{\tau}(x_0)).\]
Hence \eqref{e14f} yields
\begin{equation*}
\begin{aligned}
(-\Delta)^{-1}\psi(x)&\geq \frac{c_1}{c_3}\|\psi\|_1 \int_{d(x, x_0)}^{\infty}\frac{\tau}{\Vol(B_{\tau}(x_0))}d\tau\\
&\geq \frac{c_1}{c_3 c_2}\|\psi\|_1G(x, x_0)\,.
\end{aligned}
\end{equation*}
Let us now deal with case $(b)$. Due to Lemma \ref{Greenest}-(i) and the fact that
\[d(x,y)\leq d(x, x_0)+d(x_0, y)\leq 1+\sigma \leq 2 \quad \text{ for any }\;y\in B_\sigma (x_0),\]
we get
\begin{equation*}\label{e16f}
\begin{aligned}
(-\Delta)^{-1}\psi(x) &\geq \frac{c_1}{(n-2)\omega_n}\int_{B_\sigma(x_0)}\frac{\psi(y)}{d(x,y)^{n-2}}d\mu(y)\\
&\geq \frac{c_1}{(n-2)\omega_n 2^{n-2}}\|\psi\|_1\,.
\end{aligned}
\end{equation*}
Now note that by Lemma \ref{Greenest}-(iii), since $d(x,x_0)<1\leq R_0$, we have
$$(d(x,x_0))^{n-2}G(x,x_0)\leq \frac{c_2}{\alpha}\left[\frac{R_0^n}{n-2}+\gamma\beta f(R_0)R_0^{n-1}\right].$$
Consequently, for $d(x,x_0)<1$,
\begin{equation*}\label{e21f}
\begin{aligned}
(-\Delta)^{-1}\psi(x)&\geq  \frac{c_1 \|\psi\|_1}{(n-2)\omega_n 2^{n-2}}\\
& \geq \frac{c_1\alpha \|\psi\|_1 G(x, x_0)(d(x, x_0))^{n-2}}{(n-2)\omega_n 2^{n-2}c_2}\left(\gamma\beta R_0^{n-1} f(R_0) + \frac{R_0^n}{(n-2)} \right)^{-1}\,.
\end{aligned}
\end{equation*}
It follows that for every $x\in M$
\begin{equation*}\label{e22f}
\begin{aligned}
(-\Delta)^{-1}\psi(x)\geq \gamma_1 \|\psi\|_1\left(d(x, x_0)^{n-2}\wedge 1 \right) G(x, x_0),
\end{aligned}
\end{equation*}
with
\[\gamma_1=\frac{c_1}{c_2 c_3}\wedge\frac{c_1\alpha}{(n-2)\omega_n 2^{n-2}c_2}\left(\gamma\beta R_0^{n-1} f(R_0)+\frac{R_0^n}{(n-2)} \right)^{-1}\,.\]
\end{proof}

\begin{remark}\label{remsigma} Let $\sigma_0>1.$
It is direct to see that by minor changes in the proof of Proposition \ref{stimapr}, the estimate from above in \eqref{e1f} still holds whenever $\sigma\in (0, \sigma_0]$, with the modified constant  \[ \gamma_2=\frac{c_2 c_3^3}{c_1}+\frac{9(n-2) 2^{n-3}c_2\sigma_0^n\omega_n}{\alpha c_1}.\]
\end{remark}

\begin{lemma}\label{lemma1f}
Let assumptions of Proposition \ref{stimapr} be satisfied. Let $x_0\in M$. Then there exists $\psi\in L^\infty_c(M), \psi\geq 0, \psi\not\equiv0$ and two constants $k_1, k_2>0$, only depending on $c_1, c_2, c_3, \alpha, \beta, \gamma, n, R_0$, such that, for all $\varphi\in L^1_{x_0, G}(M), \varphi\geq 0$, we have
\begin{equation*}\label{e25bf}
k_1\int_M \varphi(-\Delta)^{-1} \psi\, d\mu \leq \|\varphi\|_{L^1_{x_0, G}} \leq k_2 \int_M \varphi(-\Delta)^{-1}\psi\,{\rm d}\mu\,.
\end{equation*}
\end{lemma}
\begin{proof}
We choose $\psi=\frac{1}{\Vol\big(B_{\frac{1}{2}}(x_0)\big)}\chi_{B_{\frac{1}{2}}(x_0)}$. By virtue of Proposition \ref{stimapr}, since $$\Vol(B_{\frac 1 2}(x_0))\|\psi\|_\infty=\|\psi\|_1=1,$$ we immediately get
\begin{equation}\label{e26bf}
\begin{aligned}
\gamma_1\int_{{M\setminus B_1(x_0)}}\varphi(x) G(x,x_0)\,{\rm d}\mu(x)&\leq\int_{M\setminus B_1(x_0)} \varphi(x) (-\Delta)^{-1}\psi(x) d\mu(x)\\
 &\leq \gamma_2 \int_{M\setminus B_1(x_0)}\varphi(x) G(x, x_0) d\mu(x).
\end{aligned}
\end{equation}
On the other hand, by the lower bound in Proposition \ref{stimapr} and by Lemma \ref{Greenest}-i), for every $x\in B_1(x_0)$ we have
\begin{equation}\label{3}
(-\Delta)^{-1}\psi(x)\geq\gamma_1 (d(x,x_0))^{n-2}G(x,x_0)\geq\frac{\gamma_1 c_1}{(n-2)\omega_n}\,.
\end{equation}
Note also that for every $x\in B_1(x_0)$ we have $B_\frac{1}{2}(x_0)\subset B_\frac{3}{2}(x)$ and that by \eqref{volume2} we have
\begin{equation*}
\Vol(B_{\frac 1 2}(x_0))\geq \alpha \frac 1{2^n}\,.
\end{equation*}
Then for every $x\in B_1(x_0)$
\begin{align*}
(-\Delta)^{-1}\psi(x)&=\int_{M}\psi(y)G(y,x)\,{\rm d}\mu(y)\\
&=\frac{1}{\Vol(B_\frac{1}{2}(x_0))}\int_{B_\frac{1}{2}(x_0)}G(y,x)\,{\rm d}\mu(y)\leq \frac{2^n}{\alpha}\int_{B_\frac{3}{2}(x)}G(y,x)\,{\rm d}\mu(y)\,.
\end{align*}
Now if $R_0>\frac{3}{2}$ by Lemma \ref{Greenest}-iv) we deduce
\begin{equation}\label{1}
(-\Delta)^{-1}\psi(x)\leq\frac{9\omega_n c_2 n2^{n-3}}{\alpha^2}\left[\frac{R_0^n}{n-2}+\gamma\beta f(R_0)R_0^{n-1}\right],
\end{equation}
while if $R_0\leq\frac{3}{2}$ by Lemma \ref{Greenest}-v) we obtain
\begin{equation}\label{2}
(-\Delta)^{-1}\psi(x)\leq\frac{2^nc_2}{\alpha}\max\{\gamma,\tfrac{1}{2}\}h\left(\tfrac{3}{2}\right)\,.
\end{equation}
Putting together \eqref{3}, \eqref{1}, \eqref{2} we see that for every $x\in B_1(x_0)$ we have
$$
C_1\leq(-\Delta)^{-1}\psi(x)\leq C_2
$$
for some $C_1, C_2>0$ as in the statement. Then
\begin{equation}\label{4}
C_1\int_{B_1(x_0)} \varphi(x) d\mu(x)\leq \int_{B_1(x_0)} \varphi(x) (-\Delta)^{-1}\psi(x) d\mu(x) \leq  C_2\int_{B_1(x_0)} \varphi(x) d\mu(x) \,.
\end{equation}
Recalling the definition of $\|\varphi\|_{L^1_{x_0, G}(M)}$, see \eqref{e24f}, combining \eqref{e26bf} and \eqref{4} the thesis easily follows.
\end{proof}

\section{Existence of approximating solutions}\label{approx}
In this Section we prove existence of WDS to problem \eqref{e23f}, when the initial datum $u_0\in L^1(M)\cap L^\infty(M)$.
In fact, in the proof of the general existence result in Theorem \ref{exists} we will exploit such solutions to construct an approximating sequence which will converge to the WDS of \eqref{e23f} with $u_0\in L^1_G(M)$.

At first, we have the following existence result for {\it mild} solutions (see e.g. \cite{BBGM} for the standard definition), namely solutions in the semigroup sense. It is a consequence of \cite[Proposition 5.2]{BBGM}.
\begin{proposition}\label{eximildsol}
Let $u_0\in L^1(M)\cap L^\infty(M), u_0\geq 0$. Then there exists a unique nonnegative mild solution $u\in C((0,+\infty);L^1(M))$ to problem \eqref{e23f}. Such solution $u$ satisfies the following {\em monotonicity property}
\[t\mapsto t^{\frac1{m-1}}u(x,t) \quad \text{ is nondecreasing for a.e.}  x\in M,\]
and the following $L^p(M)-${\em nonexpansivity property}:
\[\|u(t)\|_{L^p(M)}\leq \|u_0\|_{L^p(M)} \quad \text{ for all } t\geq 0 \text{ and for any }\; 1\leq p\leq \infty\,.\]
In addition, let $v_0\in L^1(M)\cap L^\infty(M), v_0\geq 0$ and $v\in C((0, +\infty); L^1(M))$ be the mild solution to problem \eqref{e23f} corresponding to the initial datum $v_0$. Then
\[\|u(t)-v(t)\|_{L^1(M)}\leq \|u_0-v_0 \|_{L^1(M)} \quad \text{ for all }\; t\geq 0\,.\]
\end{proposition}

We aim at proving that the mild solution is also a WDS to \eqref{e23f}. In order to do this, we need some preliminary results.

Let  $\{T_t\}_{t\ge0}$ be the heat semigroup, which is strongly continuous in $L^p(M)$ for every $p\in [1, \infty)$, and satisfies 
\[\lim_{h\to 0}\frac{u- T_h u}{h}=(-\Delta) u \quad \text{ in }\; L^p(M)\]
for every $p\in [1, \infty)$, since $-\Delta$ is the generator of $\{T_t\}$.

\begin{lemma}\label{lemmaMM}
 Let $M$ be a $n$-dimensional Riemannian manifold with $\Ric\geq0$ and assume that conditions \eqref{noncollapsing}, \eqref{uniformvolume} and \eqref{integrablef} hold. Then
$(-\Delta)^{-1}:L^1(M)\cap L^\infty(M)\to L^\infty(M)$ is continuous. Furthermore,
\begin{equation}\label{e60f}
\lim_{r\to \infty} \int_0^r T_t u\, dt = \int_0^\infty T_t u\, dt=  (-\Delta)^{-1}u  \quad \text{ in } L^\infty(M)\,.
\end{equation}
\end{lemma}
\begin{proof}
Let $ K(t,x,y)$ is the (minimal) heat kernel. By \cite[Corollary 3.1]{LY} there exists a constant $C>0$ such that
\begin{equation}\label{e61f}
K(x,y,t)\leq \frac{C}{\Vol(B_{\sqrt t}(x))}\quad \text{ for any } x,y\in M, t>0\,.
\end{equation}
By virtue of \eqref{volume2},
\begin{equation*}\label{e63f}
\Vol(B_{\sqrt t}(x))\geq \frac{\alpha}{R_0^{n}}t^{\frac n2}\quad \text{ for all }\; x\in M, 0<t< R_0^2\,.
\end{equation*}
On the other hand, in view of \eqref{uniformvolume} and \eqref{noncollapsing},
\begin{equation}\label{e62f}
\Vol(B_{\sqrt t}(x)))\geq \frac{\alpha}{\gamma R_0 f(R_0)}\sqrt{t} f(\sqrt t)\quad \text{ for all }\; x\in M, t\geq  R_0^2\,.
\end{equation}
Putting together \eqref{e61f}-\eqref{e62f} we can infer that for some $\hat C>0$
\begin{equation}\label{e64f}
K (x,y,t)\leq \frac{\hat C}{\phi(t)} \quad \text{ for all }\; x,y\in M, t>0,
\end{equation}
where
\[\phi(t):=\begin{cases}
t^{\frac n2} & \textrm{ for any }\; 0<t<R_0^2\\
\sqrt t f(\sqrt t) & \textrm{ for any }\; t\geq R_0^2\,.
\end{cases}
\]
Now we show that $(-\Delta)^{-1}:L^1(M)\cap L^\infty(M)\to L^\infty(M)$ is continuous. Indeed, from \eqref{e64f} and \eqref{integrablef} we get
\begin{equation*}\label{e65f}
\begin{aligned}
\|(-\Delta)^{-1} u\|_\infty &= \left\|\int_0^\infty T_t u\, dt \right\|_\infty \leq \int_0^{R_0^2} \left\| T_t u \right\|_\infty \, dt + \int_{R_0^2}^\infty \left\| T_t u \right\|_\infty \, dt\\
& \leq R_0^2 \|u\|_\infty + \int_{R_0^2}^\infty \left\|T_t u \right\|_\infty dt  \leq R_0^2 \|u\|_\infty + \hat{C}\|u\|_1 \int_{R_0^2}^\infty \frac 1{\phi(t)} dt \\
& =  R_0^2 \|u\|_\infty + 2\beta\hat{C}  \|u\|_1 \leq \max\{R_0^2, 2\beta\hat{C}\} (\|u\|_\infty+\|u\|_1 )\,,
\end{aligned}
\end{equation*}
which yields the desired continuity. In order to obtain \eqref{e60f}, note that for every $r>R_0^2$,
\begin{equation*}
\left\|\int_0^r T_t u\, dt -\int_0^\infty T_t u\, dt \right\|_\infty \leq \hat{C}\|u\|_1 \int_r^\infty \frac 1{\phi(t)}\, dt=2\hat{C}\|u\|_1 \int_{\sqrt{r}}^\infty \frac 1{f(s)}\, ds \to 0
\end{equation*}
as $r$ tends to $\infty$, by \eqref{integrablef}. Therefore,
\[
\lim_{r\to \infty} \int_0^r T_t u\, dt =\int_0^\infty T_t u \, dt = (-\Delta)^{-1} u \quad \text{ in }\; L^\infty(M)\,.
\]
which is \eqref{e60f}.
\end{proof}

\begin{lemma}\label{lemmaMM2}
Let $M$ be a $n$-dimensional Riemannian manifold with $\Ric\geq0$ and assume that conditions \eqref{noncollapsing}, \eqref{uniformvolume} and \eqref{integrablef} hold.
Suppose that $u, (-\Delta)u\in L^1(M)\cap L^\infty(M).$ Then
\begin{equation}\label{e66f}
(-\Delta)^{-1}(-\Delta u) = u \quad \text{ in }\; L^1(M)\cap L^\infty(M)\,.
\end{equation}
\end{lemma}
\begin{proof} We adapt to the present situation the arguments of \cite[Proposition 6.2.12]{Jacob} (see also \cite[Lemma 5.2]{BBGM})\,.
Observe that for any $r>R_0^2$ we have
\[\int_0^r T_t (u- T_h u)\, dt = \int_0^h T_t u\, dt - \int_r^{r+h} T_t u\, dt\,.\]
Dividing by $h>0$ and passing to the limit as $h\to 0^+$ we get, for any $p\in [1,\infty)$,
\begin{equation}\label{e67f}
\lim_{h\to 0}\int_0^r T_t \left(\frac{u- T_h u}{h} \right)\, dt = \int_0^r T_t (-\Delta u)\, dt \quad \text{ in }\; L^p(M)
\end{equation}
and
\begin{equation}\label{e68f}
\lim_{h\to 0}\frac 1 h\int_0^h T_t u\, dt = u \quad \text{ and } \; \lim_{h\to 0} \frac 1h \int_r^{r+h} T_t u\, dt = T_r u \quad \text{ in }\; L^p(M);
\end{equation}
here we used the continuity of $u\mapsto \int_0^r T_t u \, dt$ in $L^q(M)$ for any $q\in [1, +\infty]$\,.
From \eqref{e67f} and \eqref{e68f} we obtain
\[\int_0^r T_t(-\Delta u)\, dt =  u - T_r u\,.\]
By virtue of \eqref{e64f} and \eqref{e60f}, passing to the limit as $r\to +\infty$  we obtain
\[\int_0^\infty T_t (-\Delta u)\, dt = \lim_{r\to \infty} (u- T_r u) = u \quad \text{ in } L^\infty(M)\,,\]
that gives \eqref{e66f}.
\end{proof}

We can obtain the following existence result of WDS by arguing as in the proof of \cite[Proposition 5.3]{BBGM}, taking advantage of Lemma \ref{lemmaMM2}.
\begin{proposition}\label{exiWDS}
Let $M$ be a $n$-dimensional Riemannian manifold with $\Ric\geq0$ and assume that conditions \eqref{noncollapsing}, \eqref{uniformvolume} and \eqref{integrablef} hold. Let $u_0\in L^1(M)\cap L^\infty(M), u_0\geq 0$. Let $u$ be the mild solution to \eqref{e23f} provided by Proposition \ref{eximildsol}. Then $u$ is a WDS to \eqref{e23f} in the sense of Definition \ref{defsol}.
\end{proposition}

\begin{lemma}\label{lemma2f}
Let assumptions of Proposition \ref{stimapr} be satisfied. Let $u_0\in L^1(M)\cap L^\infty(M), u_0\geq 0$ and $u$ be a WDS to \eqref{e23f}. Then, for every $x_0\in M$,
\begin{equation}\label{e40f}
\|u(t)\|_{L^1_{x_0, G}}\leq C \|u_0\|_{L^1_{x_0, G}}\quad \text{ for all }\; t\geq 0,
\end{equation}
for some $C>0$, only depending on $c_1, c_2, c_3, \alpha, \beta, \gamma, n, R_0$. Furthermore, if $u$ and $v$ are two WDS of \eqref{e23f} corresponding to initial data $u_0$ and $v_0$, respectively, with $u\geq v$, then, for every $x_0\in M$,
\begin{equation}\label{e41f}
\|u(t)-v(t)\|_{L^1_{x_0, G}}\leq C \|u_0-v_0\|_{L^1_{x_0, G}}\quad \text{ for all }\; t\geq 0.
\end{equation}
In addition, for any $0<R\leq 1$ and $x_0\in M$, we have
\begin{equation}\label{e42f}
R^{n-2}\int_{M\setminus B_R(x_0)}u(x,t) G(x, x_0)\, d\mu(x) \leq \tilde C \|u(t)\|_{L^1_{x_0, G}}\quad \text{ for all }\; t\geq 0,
\end{equation}
for some $\tilde C>0$, only depending on $c_1, c_2, c_3, \alpha, \beta, \gamma, n, R_0$.
\end{lemma}

\begin{proof}
In view of Lemma \ref{lemma1f}, \eqref{e40f} and \eqref{e41f} can be shown arguing as in the proof of \cite[Proposition 5.5]{BBGM}.
It remains to show \eqref{e42f}. To this purpose, for every fixed $x_0\in M$, let $\psi=\frac{1}{\Vol\big(B_{\frac{1}{2}}(x_0)\big)}\chi_{B_{\frac{1}{2}}(x_0)}$, as in the proof of Proposition \ref{lemma1f}. Since $u\geq0$, from Proposition \ref{stimapr} we deduce that, for every $R\in (0, 1]$,
\begin{equation}\label{e43f}
\begin{aligned}
\int_M u(x, t) (-\Delta)^{-1}\psi(x)\, d\mu(x) & \geq \gamma_1 \int_{M\setminus B_R(x_0)}(d(x,x_0)^{n-2}\wedge1)u(x, t) G(x, x_0)\, d\mu(x)\\
&\geq \gamma_1 R^{n-2} \int_{M\setminus B_R(x_0)} u(x, t) G(x, x_0)\, d\mu(x)\,.
\end{aligned}
\end{equation}
On the other hand, since $u(t)\in L^1_{x_0,G}(M)$ for every $t\geq0$ by \eqref{e40f}, by virtue of Lemma \ref{lemma1f} we obtain
\begin{equation}\label{e44f}
k_1 \int_M u(x, t) (-\Delta)^{-1}\psi(x)\, d\mu(x)\leq  \|u(t)\|_{L^1_{x_0, G}}
\end{equation}
for every $t\geq0$. Putting together \eqref{e43f} and \eqref{e44f} the thesis follows.
\end{proof}

\section{Basic estimates for Weak Dual Solutions}\label{basic}

We now prove some crucial estimates for WDS, in the case $u_0\in L^1(M)\cap L^\infty(M)$, which are the analogue of similar bounds proved in \cite{BBGM, BE, BV}.

\begin{proposition}\label{prop53}
Let $M$ be a $n$-dimensional Riemannian manifold with $\Ric\geq0$ and assume that conditions \eqref{noncollapsing}, \eqref{uniformvolume} and \eqref{integrablef} hold. Let $u_0\in L^1(M)\cap L^\infty(M), u_0\geq 0$. Let $u$ be the nonnegative WDS to \eqref{e23f}. Then
\begin{equation}\label{e26f}
\int_M u(x,t)\, G(x, x_0)\, d\mu(x) \leq \int_M u_0(x)\, G(x, x_0)\, d\mu(x) \quad \text{ for all } x_0\in M, t\geq 0\,,
\end{equation}
and
\begin{equation}\label{e27f}\begin{aligned}
\left(\frac{t_0}{t_1} \right)^{\frac m{m-1}}(t_1-t_0) u^m(t_0, x_0)&\leq \int_M[u(x, t_0)-u(x, t_1)]G(x, x_0)\, d\mu(x) \\ &\leq (m-1)\frac{t^{\frac m{m-1}}}{t_0^{\frac 1{m-1}}}u^m(x_0, t)
\end{aligned}\end{equation}
for a.e. $t_0, t_1, t\in \mathbb R$ with $0<t_0\leq t_1\leq t$ and for a.e. $x_0\in M\,.$
\end{proposition}

\begin{proof}
We follow the line of arguments of \cite[Proposition 5.4]{BBGM}, where the fractional Laplacian is treated on Riemannian manifolds with Ricci bounded from below and fulfilling a certain Faber-Krahn inequality (that now is not guaranteed).

It is easily seen that
\begin{equation}\label{e28f}
\begin{aligned}
\int_M u(x, t_1)(-\Delta)^{-1}\eta(x) \,{\rm d}\mu(x) \leq \int_M u(x, t_0) (-\Delta)^{-1}\eta(x)\, d\mu(x)
\end{aligned}
\end{equation}
for all $\eta\in L^\infty_c(M)$, $\eta\geq 0$ and for all $0<t_0<t_1\,.$ In fact, for any $n\in \mathbb N,$ we insert in \eqref{e25f} the test function $\psi(x,t)=\phi_n(t)\eta(x),$ where $\{\phi_n\}\subset C^1_c((0, \infty))$ is a sequence fulfilling $\phi_n(t) \to \chi_{[t_0, t_1]}(t)$ a.e. in $(0, \infty)$ and $\phi_n'(t) \to \delta_{t_0}-\delta_{t_1},$ as $n\to \infty.$
We now show the continuity of $t\mapsto (-\Delta)^{-1}u(\cdot, t)$. In fact, let $\sigma_0>0$, $\sigma\in (0, \sigma_0], x_0\in M$ and $\psi:=\chi_{B_{\sigma}(x_0)}$. Then for all $x\in B_1(x_0)$, since $B_\sigma(x_0)\subset B_{\sigma+1}(x)$, we have
\begin{equation}\label{e46f}
\begin{aligned}
(-\Delta)^{-1}\psi(x)&=\int_{M} G(x, y) \psi(y)\, d\mu(y) \\ &\leq\int_{B_{\sigma+1}(x)} G(x, y) \, d\mu(y)\\ & \leq \begin{cases}
\frac{\omega_n c_2 n}{2\alpha}\left[\frac{R_0^n}{n-2}+\gamma\beta f(R_0)R_0^{n-1}\right](\sigma+1)^2 & \text{ if } \sigma<R_0-1\\
c_2\max\{\gamma,\tfrac{1}{2}\} h(\sigma+1)  & \text{ if } \sigma\geq R_0-1\,,
\end{cases}
\end{aligned}
\end{equation}
with $h$ defined as in \eqref{5}, see Lemma \ref{Greenest}. On the other hand, by virtue of Proposition \ref{stimapr} and Remark \ref{remsigma}, for all $x\in M\setminus B_1(x_0)$,
\begin{equation}\label{e47f}
(-\Delta)^{-1}\psi(x)\leq \gamma_2(\sigma_0)\operatorname{Vol}(B_\sigma(x_0))G(x, x_0).
\end{equation}
In view of \eqref{e46f} and \eqref{e47f}, by the same arguments as in \cite[Remark 2.1]{BBGM}, we can infer that $(-\Delta)^{-1}u\in C^0([0, T]; L^1_{\textrm{loc}}(M))$.

By using the continuity of $t\mapsto (-\Delta)^{-1}u(\cdot, t)$  and Fubini's theorem, and letting $n\to \infty,$ we obtain \eqref{e28f}.

Now, fix an arbitrary point $x_0\in M$. For every $n\in \mathbb N,$ let
\[\eta_n^{x_0}\equiv \eta_n:=\frac{\chi_{B_{\frac 1n}(x_0)}}{\Vol(B_{\frac 1n}(x_0))}\quad \text{ in } M\,.\]
We claim that
\begin{equation}\label{e33f}
\left|\int_{M}u(x, \tau)(-\Delta)^{-1}\eta_n(x)\, d\mu(x) -\int_{M }u(x, \tau)G(x, x_0)\,{\rm d}\mu(x) \right|\to 0,
\end{equation}
for a.e. $\tau>0$ as $n\to \infty.$ By means of \eqref{e33f}, \eqref{e26f} easily follows. In fact, since, for any $n\in \mathbb N,$ $\eta_n\geq 0$, $\eta_n\in L^\infty_c(M)$, we can use $\eta_n$ as test function in \eqref{e28f}. Hence, by setting $t:=t_1$,
\begin{equation}\label{e34f}
\begin{aligned}
\int_M u(x, t)(-\Delta)^{-1}\eta_n(x)\, d\mu(x) \leq \int_M u(x, t_0) (-\Delta)^{-1}\eta_n(x)\, d\mu(x)
\end{aligned}
\end{equation}
for all $0<t_0<t\,.$ In view of \eqref{e33f}, by letting $n\to \infty$ in \eqref{e34f} and then $t_0\to 0^+$ we get \eqref{e26f}.

It remains to show \eqref{e33f}. To this purpose, recall that for each $x\in M$, $y\mapsto G(x, y)$ is continuous in $M\setminus\{x\}$; furthermore, $\eta_n\to \delta_{x_0}$ in the sense of Radon measure, as $n\to \infty.$ Consequently, for any $x\in M\setminus\{x_0\}$,
\begin{equation}\label{e29f}
(-\Delta)^{-1}\eta_n(x)=\frac1{\Vol(B_{\frac 1n}(x_0))}\int_{B_{\frac 1n}(x_0)} G(x, y) \,{\rm d}\mu(y)\to G(x, x_0),
\end{equation}
as $n\to +\infty$. Moreover, in view of \eqref{e1f},
\begin{equation}\label{6}
0\leq (-\Delta)^{-1}\eta_n(x) \leq \gamma_2 G(x, x_0)\quad \text{ for all } x\in M\setminus\{x_0\}, n\in \mathbb N\,.
\end{equation}
Thus,
\begin{equation}\label{e31f}
u(x, \tau) (-\Delta)^{-1}\eta_n(x) \leq \gamma_2 u(x, \tau) G(x, x_0)\quad \text{ for all } x\in M\setminus\{x_0\}, n\in \mathbb N\,.
\end{equation}
Let $R >0$. Since $u(\cdot, \tau)\in L^1_{x_0, G}(M)$, thanks to \eqref{e31f} and the dominated convergence theorem we deduce that
\begin{equation}\label{e32f}
\left|\int_{M\setminus B_R(x_0)}u(x, \tau)(-\Delta)^{-1}\eta_n(x)\, d\mu(x) -\int_{M\setminus B_R(x_0)}u(x, \tau)G(x, x_0)\,{\rm d}\mu(x) \right|\to 0,
\end{equation}
as $n\to \infty.$

On the other hand, due to Lemma \ref{Greenest}, for any $R>0$ and $p\in \left[1, \frac{n}{n-2}\right)$, $G(\cdot, x_0)\in L^p(B_R(x_0))$. Therefore, by virtue of \eqref{e29f}, \eqref{6} and Lebesgue's convergence theorem,
\[(-\Delta)^{-1}\eta_n(x)\to G(x, x_0) \quad \text{ in } L^p(B_R(x_0)), \text{ as } n\to \infty\,,\]
for all $p\in \left[1, \frac{n}{n-2}\right)$. Hence, for any $R>0$ and $\tau>0$,
\begin{equation}\label{e30f}
\begin{aligned}
&\left|\int_{B_R(x_0)}u(x, \tau)(-\Delta)^{-1}\eta_n(x)\, d\mu(x) -\int_{B_R(x_0)}u(x, \tau)G(x, x_0)\,{\rm d}\mu(x) \right|\\
&\leq \|u(\tau)\|_{L^{p'}(M)}\|(-\Delta)^{-1}\eta_n - G(\cdot, x_0)\|_{L^p(B_R(x_0))}\to 0 \quad \text{ as } n\to \infty,
\end{aligned}
\end{equation}
$p'$ being the conjugate exponent of $p$. Putting together \eqref{e32f} and \eqref{e30f}, we obtain \eqref{e33f}, and the proof of \eqref{e26f} is complete.

\smallskip

Finally, \eqref{e27f} can be obtained by minor modifications in the proof of \cite[Proposition 3.3]{BBGG}.
\end{proof}

\begin{remark}\label{worst}
  The first inequality in \eqref{e27f}, and \eqref{e26f}, allow to show that, setting $t_0=t$, $t_1=2t_0$,
  \[
  u(t,x_0)\le \frac c{t^{\frac1m}} \left(\int_M u(x, t)G(x, x_0)\, d\mu(x)\right)^{\frac1m}\le \frac c{t^{\frac1m}} \left(\int_M u_0(x)G(x, x_0)\, d\mu(x)\right)^{\frac1m},\ \ \forall t>0,
  \]
 since the latter integral exists under our assumptions. If this holds uniformly in $x_0$, clearly we directly have a smoothing effect with time decay $t^{-\frac1m}$. Of course this is a weaker conclusion than the one of Theorem \ref{decayL1G}, in view of the definition of the $L^1_G(M)$ norm.

\end{remark}


\section{Proof of the Main Results}\label{proofmain}

\begin{proof}[Proof of Theorem \ref{decay}]
  It is standard to show that if \eqref{7} and \eqref{13} hold for $u_0\in L^1(M)\cap L^\infty(M)$, $u_0\geq 0$, then by an approximation argument (through a monotone sequence of functions converging to the initial datum in $L^1(M)$) they also hold for a general $u_0\in L^1(M)$, $u_0\geq 0$, see e.g. \cite[Theorem 2.6]{BBGM}. Without loss of generality we will then restrict our attention to the case $u_0\in L^1(M)\cap L^\infty(M)$, $u_0\geq 0$.

  From \eqref{e27f}, with $t_1=2t_0$, for a.e. $x_0\in M$, $t_0>0$  we have
  \begin{align}
  \nonumber u^{m}(x_0,t_0)&\leq \frac{C}{t_0}\int_Mu(x,t_0)G(x,x_0)\,{\rm d}\mu(x)\\
  \label{11}  &=C\left(\frac{1}{t_0}\int_{B_R(x_0)}u(x,t_0)G(x,x_0)\,{\rm d}\mu(x)+\frac{1}{t_0}\int_{M\setminus B_R(x_0)}u(x,t_0)G(x,x_0)\,{\rm d}\mu(x)\right)\\
  \nonumber  &=:C\left(I+J\right)
  \end{align}
  for any $R\geq R_0$. By Young's inequality for every $\varepsilon>0$ we have
  \begin{align}
  \label{10}I&\leq \frac{\varepsilon}{m}\|u(t_0)\|_{L^\infty(M)}^m+\frac{C_\varepsilon}{t_0^\frac{m}{m-1}}\left(\int_{B_R(x_0)}G(x,x_0)\,{\rm d}\mu(x)\right)^\frac{m}{m-1}\\
   \nonumber &\leq\frac{\varepsilon}{m}\|u(t_0)\|_{L^\infty(M)}^m+\frac{C_\varepsilon}{t_0^\frac{m}{m-1}}\left(h(R)\right)^\frac{m}{m-1}\,,
  \end{align}
  where in the last inequality we used \eqref{5}. On the other hand by \eqref{8} we have
  \begin{align*}
  J&\leq \frac{1}{t_0}\|u(t_0)\|_{L^1(M)}\sup_{x\in M\setminus B_R(x_0)}G(x,x_0)\\
  &\leq \frac{\gamma c_2}{t_0}\|u(t_0)\|_{L^1(M)}\sup_{\rho\geq R} \frac{\rho f(\rho)}{\Vol(B_{\rho}(x_0))}\int_{\rho}^{+\infty}\frac{1}{f(t)}\,{\rm d}t\\
  &\leq \frac{\gamma^2 c_2}{t_0}\|u(t_0)\|_{L^1(M)}\frac{Rf(R)}{\Vol(B_R(x_0))}\int_{R}^{+\infty}\frac{1}{f(t)}\,{\rm d}t\,,
  \end{align*}
where in the last inequality we used \eqref{uniformvolume}. Then by the definition of $h,F$ and by Proposition \ref{eximildsol} we have
\begin{equation}\label{9}
  J\leq C\frac{h(R)}{t_0F(R)}\|u_0\|_{L^1(M)}\,.
\end{equation}
Plugging \eqref{10} and \eqref{9} into \eqref{11} we obtain
\[
u^{m}(x_0,t_0)\leq \varepsilon\frac{C}{m}\|u(t_0)\|_{L^\infty(M)}^m+C_\varepsilon\left(\frac{1}{t_0^\frac{m}{m-1}}\left(h(R)\right)^\frac{m}{m-1}+\frac{h(R)}{t_0F(R)}\|u_0\|_{L^1(M)}\right)
\]
and therefore choosing $\varepsilon>0$ small enough
\begin{equation}\label{12}
\|u(t_0)\|_{L^{\infty}(M)}\leq \frac{C}{t_0^\frac{1}{m-1}}(h(R))^\frac{1}{m-1}\left(1+\frac{t_0^\frac{1}{m-1}}{\theta(R)}\|u_0\|_{L^1(M)}\right)^\frac{1}{m}.
\end{equation}
Now if $t_0\geq\left(\frac{\theta(R_0)}{\|u_0\|_{L^1(M)}}\right)^{m-1}$ we can choose $R=\theta^{-1}(t_0^{\frac{1}{m-1}}\|u_0\|_{L^1(M)})\geq R_0$ in \eqref{12}, and thus we obtain \eqref{7}, with $K=\theta(R_0)^{m-1}$.

In order to prove \eqref{13} we start from \eqref{11}, with $0<R\leq1$. Then by Young's inequality for every $\varepsilon>0$ we have
  \begin{align}
  \label{16}I&\leq \frac{\varepsilon}{m}\|u(t_0)\|_{L^\infty(M)}^m+\frac{C_\varepsilon}{t_0^\frac{m}{m-1}}\left(\int_{B_R(x_0)}G(x,x_0)\,{\rm d}\mu(x)\right)^\frac{m}{m-1}\\
   \nonumber &\leq\frac{\varepsilon}{m}\|u(t_0)\|_{L^\infty(M)}^m+\frac{C_\varepsilon}{t_0^\frac{m}{m-1}}R^\frac{2m}{m-1}\,,
  \end{align}
where in the last inequality we used \eqref{14}. From \eqref{e42f} and Proposition \ref{eximildsol} we have
\begin{align}\label{17}
J&\leq \frac{C}{t_0R^{n-2}}\|u(t_0)\|_{L^1_{x_0, G}}\leq \frac{C}{t_0R^{n-2}}\|u(t_0)\|_{L^1(M)}\leq \frac{C}{t_0R^{n-2}}\|u_0\|_{L^1(M)}
\end{align}
since $G(x,x_0)$ is uniformly bounded for $d(x,x_0)\geq1$, see Remark \ref{15}. Plugging \eqref{16} and \eqref{17} into \eqref{11} we obtain
\[
u^{m}(x_0,t_0)\leq \varepsilon\frac{C}{m}\|u(t_0)\|_{L^\infty(M)}^m+C_\varepsilon\left(\frac{1}{t_0^\frac{m}{m-1}}R^\frac{2m}{m-1}+\frac{1}{t_0R^{n-2}}\|u_0\|_{L^1(M)}\right)
\]
and hence for small enough $\varepsilon>0$
\begin{equation*}
\|u(t_0)\|_{L^{\infty}(M)}\leq \frac{C}{t_0^\frac{1}{m-1}}R^\frac{2}{m-1}\left(1+\frac{t_0^\frac{1}{m-1}}{R^{\frac{n(m-1)+2}{m-1}}}\|u_0\|_{L^1(M)}\right)^\frac{1}{m}.
\end{equation*}
If $t_0\leq K\|u_0\|_{L^1(M)}^{-(m-1)}$ we can choose $R=K^{-\frac{1}{n(m-1)+2}}t_0^\frac{1}{n(m-1)+2}\|u_0\|_{L^1(M)}^\frac{m-1}{n(m-1)+2}\leq1$ and we obtain
\begin{equation*}
\|u(t_0)\|_{L^{\infty}(M)}\leq \frac{C}{t_0^\frac{n}{n(m-1)+2}}\|u_0\|_{L^1(M)}^\frac{2}{n(m-1)+2},
\end{equation*}
that is \eqref{13}, again with $K=\theta(R_0)^{m-1}$ as in \eqref{7}.
\end{proof}

\begin{proof}[Proof of Theorem \ref{decayL1G}]
  It is standard to show that also in this case if \eqref{16} and \eqref{17} hold for $u_0\in L^1(M)\cap L^\infty(M)$, $u_0\geq 0$, then by an approximation argument they also hold for a general $u_0\in L^1_G(M)$, $u_0\geq 0$, see e.g. \cite[Theorem 2.7]{BBGM}. Without loss of generality we will then restrict our attention to the case $u_0\in L^1(M)\cap L^\infty(M)$, $u_0\geq 0$.

  From \eqref{11} and \eqref{10} we immediately see that for every $R\geq R_0$ and every $\varepsilon>0$ we have
  \begin{equation*}
    u^{m}(x_0,t_0)\leq \varepsilon\frac{C}{m}\|u(t_0)\|_{L^\infty(M)}^m+C_\varepsilon\left(\frac{1}{t_0^\frac{m}{m-1}}\left(h(R)\right)^\frac{m}{m-1}+\frac{1}{t_0}\int_{M\setminus B_R(x_0)}u(x,t_0)G(x,x_0)\,{\rm d}\mu(x)\right)
  \end{equation*}
  for a.e. $x_0\in M$, $t_0>0$. Choosing $R=R_0\geq1$ and recalling \eqref{e24f} we then obtain
  \begin{align*}
    u^{m}(x_0,t_0)&\leq \varepsilon\frac{C}{m}\|u(t_0)\|_{L^\infty(M)}^m+C_\varepsilon\left(\frac{1}{t_0^\frac{m}{m-1}}\left(h(R_0)\right)^\frac{m}{m-1}+\frac{1}{t_0}\|u(t_0)\|_{L^1_{x_0,G}}\right)\\
    &\leq\varepsilon\frac{C}{m}\|u(t_0)\|_{L^\infty(M)}^m+C_\varepsilon\left(\frac{1}{t_0^\frac{m}{m-1}}+\frac{1}{t_0}\|u_0\|_{L^1_{x_0,G}}\right)\,,
  \end{align*}
  where in the last inequality we used \eqref{e40f}. Choosing $\varepsilon>0$ small enough thus we have
  \begin{align*}
  \|u(t_0)\|_{L^\infty(M)}&\leq C\left(\frac{1}{t_0^\frac{m}{m-1}}+\frac{1}{t_0}\|u_0\|_{L^1_G(M)}\right)^\frac{1}{m}\\
  &= \frac{C}{t_0^\frac{1}{m}}\|u_0\|_{L^1_G(M)}^\frac{1}{m}\left(\frac{1}{t_0^\frac{1}{m-1}\|u_0\|_{L^1_G(M)}}+1\right)^\frac{1}{m}\leq \frac{C}{t_0^\frac{1}{m}}\|u_0\|_{L^1_G(M)}^\frac{1}{m}
  \end{align*}
  for every $t_0\geq\|u_0\|_{L^1_G(M)}^{-(m-1)}$, that is \eqref{20}.

  The proof of \eqref{21} is similar to that of \eqref{13} in Theorem \ref{decay}. Here one just has to observe that by \eqref{e42f} and \eqref{e40f} for a.e. $x_0\in M$, $t_0>0$ we have
  \begin{align}
\label{22}J&=\frac{1}{t_0}\int_{M\setminus B_R(x_0)}u(x,t_0)G(x,x_0)\,{\rm d}\mu(x)\\
\nonumber&\leq \frac{C}{t_0R^{n-2}}\|u(t_0)\|_{L^1_{x_0, G}}\leq \frac{C}{t_0R^{n-2}}\|u_0\|_{L^1_{x_0,G}}\leq \frac{C}{t_0R^{n-2}}\|u_0\|_{L^1_{G}}
\end{align}
for every $0<R\leq1$. Then the proof proceeds as in Theorem \ref{decay}, with \eqref{22} replacing \eqref{17}, and we obtain \eqref{21}.
\end{proof}

\section{Applications and Optimality}\label{optimal}
We now apply our main results to more specific classes of manifolds $M$ satisfying our assumptions, proving Corollary \ref{almosteuc}.
\begin{proof}[Proof of Corollary \ref{almosteuc}]  
  We start from item (a). Under the running assumptions, we have:

  \begin{align*}
    \int_R^\infty \frac{1}{f(t)}\,{\rm d}t&\asymp\frac{1}{R^{k-2}(\log(R))^\delta}, \\
    h(R) & \asymp R^2, \\
    F(R) & \geq CR^\lambda.
  \end{align*}
  Therefore $$\theta(R):=R^{\lambda+\frac{2}{m-1}}\leq CF(R)(h(R))^\frac{1}{m-1}$$ and $\theta(s)^{-1}=s^\frac{m-1}{(m-1)\lambda+2}$. If $u_0\in L^1(M)$, $u_0\geq0$, by Theorem \ref{decay}, see also Remark \ref{rem2}, we immediately deduce the assertion


  \medskip
  To deal with item (b), we notice that under the running assumptions we have:
  \begin{align*}
    \int_R^\infty \frac{1}{f(t)}\,{\rm d}t&\asymp\frac{1}{(\log(R))^{\delta-1}}, \\
    h(R) & \asymp R^2\log(R), \\
    F(R) & \geq CR^\lambda(\log(R))^\sigma.
  \end{align*}
 Then $$\theta(R):=R^{\lambda+\frac{2}{m-1}}(\log(R))^{\sigma+\frac{1}{m-1}}\leq CF(R)(h(R))^\frac{1}{m-1}$$
 and
 $$G(s):=\theta^{-1}(s)=\begin{cases}e^{\frac{b}{a}W_0\left(\frac{a}{b}s^\frac{1}{b}\right)}&\textrm{if } b\neq0\\ s^\frac{m-1}{(m-1)\lambda+2}&\textrm{if }b=0\end{cases}$$
 with
 $$a=\lambda+\frac{2}{m-1},\qquad b=\sigma+\frac{1}{m-1},$$
 where $W_0$ is the principal branch of the Lambert function, or product logarithm function, i.e. the inverse on $[-1,\infty)$ of the function $w(x)=xe^x$. If $u_0\in L^1(M)$, $u_0\geq0$, by Theorem \ref{decay}, see also Remark \ref{rem2}, we then have the assertion.
%

\end{proof}



\subsection{Optimality}\label{opt}
In this Subsection, our goal is to show that in Theorem \ref{decay}, the rate of decay of WDS with initial datum in $L^1(M)$ for large times is optimal, in general.

To this aim, let us consider the manifold $M$ as in Example \ref{exproduct} with $X=\mathbb R^k$, $k\geq3$. Then $M$ satisfies \eqref{noncollapsing}, \eqref{uniformvolume} and \eqref{integrablef} with $$f(R)=R^{k-1}$$
and $$\Vol(B_R(x))\asymp R^k$$
for every $x\in M$ and $R>0$ large enough.

For each $x\in M$, we suppose that $x\equiv (\bar x, x')$, where $\bar x\in \mathbb R^k$ and $x'$ are local coordinates in $N$.
Put $r\equiv |\bar x|_{\mathbb R^k}=\sqrt{\bar x_1^2+\ldots \bar x_k^2}$. Now, let $v$ be the Barenblatt solution to the porous medium equation in $\mathbb R^k$, that is
\[v(x,t):=\frac 1{t^{\alpha}}\left[A-\frac{c}{t^{2\beta}}r^2\right]_+^{\frac 1{m-1}},\quad x\in M, t>0,\]
where
$$\alpha=\frac{k}{k(m-1)+2}, \quad \beta=\frac{\alpha}{k}, \quad c=\frac{\alpha(m-1)}{2mk}\,,
$$
while $A>0$ is arbitrary, and $$\int_{\mathbb R^k} v(x,t)\, d\bar x=A\qquad \text{ for all }\; t>0\,.$$
Consequently, for some $\bar A>0$, depending only on $A$ and $\Vol(N)$,
\begin{equation*}\label{e50f}
\int_M v(x, t)\,{\rm d}\mu(x) = \bar A \qquad \text{ for all }\; t>0\,.
\end{equation*}

Let $\epsilon>0$. It is easily seen that $v$ is a solution of
\[v_t = \Delta v^m \quad \text{ in }\; M\times (\epsilon, \infty)\,.\]
Let
\[u(x,t):=v(x, t+\epsilon), \quad x\in M, t\geq 0\,.\]
Thus $u$ is solution of \eqref{e23f} with $u_0(x)=v(x, \epsilon)$ for any $x\in M\,.$ We can apply Theorem \ref{decay} with
\begin{align*}
h(R)&\asymp R^2,\\
\theta(R)&\asymp R^{k+\frac{2}{m-1}},\\
\theta^{-1}(s)&\asymp s^\frac{m-1}{(m-1)k+2}
\end{align*}
which ensures that
\begin{equation}\label{e51f}
\|u(t)\|_{\infty} \leq \frac{C}{t^\alpha}
\end{equation}
for all large enough $t>0$.
On the other hand, by virtue of the very definition of $u$, it is easily seen that, for some $\bar C>0$,
\begin{equation}\label{e52f}
\|u(t)\|_\infty\geq \frac {\bar C }{t^\alpha} \quad \text{ for all }\; t>1\,.
\end{equation}
Thanks to \eqref{e51f} and \eqref{e52f} we can conclude that the result in Theorem \ref{decay} on the rate of decay of WDS with initial datum in $L^1(M)$ for large times is optimal.

\bigskip
\bigskip

\noindent{\bf Acknowledgements.} The authors would like to thank M. Muratori for helpful discussions.

\noindent  The authors are members of GNAMPA-INdAM. Morever, the authors acknowledge that this work is part of the PRIN projects 2022 ''Partial differential equations and related geometric-functional inequalities", ref. 20229M52AS (GG) and ''Geometric-analytic methods for PDEs and applications", ref. 2022SLTHCE (DDM, FP), financially supported by the EU, in the framework of the "Next Generation EU initiative". DDM and FP are  partially supported by GNAMPA projects 2024.


\end{document}